\documentclass[11pt]{article}

\usepackage{mathrsfs,amssymb,amsmath,amsfonts}
\usepackage{amsthm}
\usepackage{bbm}
\usepackage{tipa}
\usepackage{enumerate}
\usepackage{color}
\usepackage[titletoc,title]{appendix}
\usepackage{cite}
\usepackage[hyphens]{url}
\usepackage[bookmarks=false,hidelinks]{hyperref}

\usepackage{stmaryrd}

\usepackage{arydshln}
\usepackage{indentfirst}
\def\sqr#1#2{{\vcenter{\vbox{\hrule height.#2pt
              \hbox{\vrule width.#2pt height#1pt \kern#1pt \vrule width.#2pt}
              \hrule height.#2pt}}}}

\def\3n{\negthinspace \negthinspace \negthinspace }
\def\2n{\negthinspace \negthinspace }
\def\1n{\negthinspace }

\def\dbE{\mathbb{E}}
\def\dbF{\mathbb{F}}

\def\dbP{\mathbb{P}}
\def\dbQ{\mathbb{Q}}
\def\dbR{\mathbb{R}}
\def\dbS{\mathbb{S}}

\def\sG{\mathscr{G}}

\def\={\buildrel \triangle \over =}

\def\ds{\displaystyle}

\def\ns{\noalign{\ss}}
\def\a{\alpha}
\def\b{\beta}

\def\d{\delta}
\def\e{\varepsilon}

\def\l{\lambda}
\def\m{\mu}

\def\si{\sigma}
\def\t{\tau}
\def\f{\varphi}
\def\th{\theta}

\def\i{\infty}
\def\D{\Delta}
\def\Th{\Theta}

\def\cB{{\cal B}}

\def\cG{{\cal G}}
\def\cH{{\cal H}}

\def\cL{{\cal L}}

\def\cS{{\cal S}}
\def\cT{{\cal T}}
\def\cU{{\cal U}}
\def\cV{{\cal V}}
\def\cW{{\cal W}}

\def\cY{{\cal Y}}
\def\cZ{{\cal Z}}

\def\BY{{\bf Y}}
\def\BZ{{\bf Z}}
\def\no{\noindent}

\def\ss{\smallskip}
\def\ms{\medskip}

\def\q{\quad}
\def\qq{\qquad}

\def\limsup{\mathop{\overline{\rm lim}}}
\def\liminf{\mathop{\underline{\rm lim}}}

\def\esssup{\mathop{\rm esssup}}
\def\essinf{\mathop{\rm essinf}}

\def\wt{\widetilde}

\def\cd{\cdot}

\def\tr{\hbox{\rm tr$\,$}}

\def\les{\leqslant}
\def\ges{\geqslant}

\def\({\Big (}
\def\){\Big )}
\def\[{\Big[}
\def\]{\Big]}
\def\bde{\begin{definition}}
\def\ede{\end{definition}}
\def\be{\begin{equation}}
\def\bel{\begin{equation}\label}
\def\ee{\end{equation}}
\def\bt{\begin{theorem}}
\def\et{\end{theorem}}
\def\bc{\begin{corollary}}
\def\ec{\end{corollary}}
\def\bl{\begin{lemma}}
\def\el{\end{lemma}}
\def\bp{\begin{proposition}}
\def\ep{\end{proposition}}
\def\bas{\begin{assumption}}
\def\eas{\end{assumption}}
\def\br{\begin{remark}}
\def\er{\end{remark}}
\def\ba{\begin{array}}
\def\ea{\end{array}}
\def\ed{\end{document}}

\def\square#1{\vbox{\hrule\hbox{\vrule height#1%
     \kern#1\vrule}\hrule}}
\def\rectangle#1#2{\vbox{\hrule\hbox{\vrule height#1%
     \kern#2\vrule}\hrule}}

\font\tenbb=msbm10 \font\sevenbb=msbm7 \font\fivebb=msbm5

\newfam\bbfam
\scriptscriptfont\bbfam=\fivebb \textfont\bbfam=\tenbb
\scriptfont\bbfam=\sevenbb

\def\md{{\mathrm d}}

\newtheorem{assumption}{Assumption}

\newtheorem{theorem}{Theorem}[section]
\newtheorem{lemma}[theorem]{Lemma}
\newtheorem{definition}[theorem]{Definition}
\newtheorem{proposition}[theorem]{Proposition}
\newtheorem{corollary}[theorem]{Corollary}
\newtheorem{remark}[theorem]{Remark}

\theoremstyle{definition}

\newenvironment{keywords}{{\bf Key words: }}{}

\numberwithin{equation}{section}

\newcommand{\udots}{\mathinner{\mskip1mu\raise1pt\vbox{\kern7pt\hbox{.}}
\mskip2mu\raise4pt\hbox{.}\mskip2mu\raise7pt\hbox{.}\mskip1mu}}

\begin{document}
\title{Stochastic Representations of Stationary HJBI-Type Variational Inequalities with Bilateral Constraints\footnote{This work is supported by the National Key R\&D Program of China (No. 2023YFA1009002),  the
National Natural Science Foundation of China (No. 12371443) and the Changbai
Talent Program of Jilin Province.}}

\author{Sheng Huang\footnote{School of Mathematics and Statistics, Northeast Normal University, Changchun 130024, P. R. China; {E-mail: huangsheng9910@163.com}}\qq
Qingmeng Wei\footnote{Corresponding author. School of Mathematics and Statistics, Northeast Normal University, Changchun 130024, P. R. China; E-mail:  weiqm100@nenu.edu.cn}}

\maketitle

\begin{abstract}
In this paper, we study probabilistic representations for stationary HJBI-type
variational inequalities with bilateral constraints. We provide two
complementary stochastic representations.
The first representation is obtained through an augmented infinite-horizon
two-player zero-sum stochastic differential game (SDG). By enlarging the
control spaces with two additional stopping symbols, the obstacle terms are
incorporated into the running payoff. Using the framework of infinite-horizon
stochastic recursive differential games, we show that the resulting lower and
upper value functions are the unique bounded viscosity solutions of the
corresponding HJBI variational inequalities.
The second representation is given by a two-player zero-sum mixed
control--stopping SDG. In this formulation, each player chooses both a
continuous control and a stopping decision, and the payoff is defined by a
BSDE with a random terminal time. To make the stopping component compatible
with the Elliott--Kalton strategy framework, we introduce nonanticipative
stopping strategies depending on the opponent's control process. The proof is
based on penalized infinite-horizon SDGs coupled with their own value
functions, together with dynamic programming arguments and stability estimates
for backward semigroups. We prove that the value functions of the mixed
control--stopping game coincide with the unique bounded viscosity solutions of
the bilateral HJBI variational inequalities.
\end{abstract}



\section{Introduction}
The Isaacs equations originate from the theory of differential games initiated
by Isaacs \cite{Isaacs-1965}.
Since value functions arising from optimal
control and differential games are in general not smooth, the corresponding
Bellman and Isaacs equations cannot usually be understood in the classical
sense. Viscosity solution theory, introduced by Crandall and Lions
\cite{Crandall-Lions-1983} and further developed for second-order fully
nonlinear equations by Crandall et al. \cite{Crandall-Ishii-Lions-1992},
provides a central weak-solution framework for studying these equations. For
deterministic differential games, viscosity solution approaches were developed
in \cite{Evans-Souganidis-1984, Lions-Souganidis-1985}. For two-player zero-sum
stochastic differential games (SDGs), Fleming and Souganidis
\cite{Fleming-Souganidis-1989} characterized the value functions as viscosity
solutions of second-order Isaacs equations.

Since the seminal work of Pardoux and Peng \cite{Pardoux-Peng-1990}, the theory of
backward stochastic differential equations (BSDEs) has been extensively
developed and has become a powerful probabilistic tool for studying nonlinear
PDEs arising from stochastic control
\cite{Peng-1991,Pardoux-Peng-1992, Peng-1997}. This BSDE-based approach was subsequently extended to SDGs by
Buckdahn and Li \cite{Buckdahn-Li-2008}, who studied two-player zero-sum SDGs
with recursive payoff functionals generated by controlled BSDEs and showed that
the   value functions are viscosity solutions of the corresponding
HJBI equations. BSDEs with random terminal times and their connections with elliptic and
parabolic PDEs have also been studied; see, for instance,
\cite{Darling-Pardoux-1997,Pardoux-1998,Royer-2004}.
Related infinite horizon recursive control and game problems have been
investigated as well; see, e.g.,
\cite{Buckdahn-Li-Zhao-2021,Li-Zhao-2019,Luo-Li-Wei-2025,Huang-Wei-2026}.

    Another closely related direction concerns HJBI variational inequalities, which
arise naturally when stopping decisions or obstacle constraints are incorporated
into   SDGs; see the early works
\cite{Bensoussan-Friedman-1974,Friedman-1975}. Stationary Isaacs-type
variational inequalities arising from mixed-type games with controls and
stopping times were subsequently studied in
\cite{Ghosh-Rao-2003,Ghosh-Rao-Sheetal-2009,Ghosh-Rao-2012}, where the
corresponding value functions were characterized as viscosity solutions of
HJI/HJBI variational inequalities.
The development of BSDE theory has also contributed to this line of research.
In particular, reflected BSDEs provide a probabilistic framework for obstacle
problems and stopping games; see, e.g.,
\cite{Hamadene-2006,ElKaroui-Kapoudjian-Pardoux-Peng-Quenez-1997,
Cvitanic-Karatzas-1996,Hamadene-Lepeltier-2000,
Hamadene-Lepeltier-Wu-1999}. Closely related to HJBI variational
inequalities, finite horizon zero-sum SDGs with payoff functionals generated by
reflected BSDEs have been studied in
\cite{Buckdahn-Li-2009,Buckdahn-Li-2011}, where the lower and upper value
functions were characterized as viscosity solutions of obstacle-type HJBI
variational inequalities, including both single-obstacle and double-obstacle
cases.
 By comparison, the BSDE-based treatment of stationary fully nonlinear
obstacle-type HJBI variational inequalities remains less developed.
This
motivates us to focus on the following lower and upper HJBI variational
inequalities:
\bel{VI-}
\min\big\{
\max\big\{
\lambda W(x)-H^-\big(x, \lambda W(x), DW(x),D^2W(x)\big),
W(x)-\psi_1(x)
\big\},
W(x)-\psi_2(x)
\big\}=0,\  x\in\mathbb R^n,
\ee
and
\bel{VI+}
\min\big\{
\max\big\{
\lambda W(x)-H^+\big(x,\lambda W(x),DW(x),D^2W(x)\big),
W(x)-\psi_1(x)
\big\},
W(x)-\psi_2(x)
\big\}=0,\  x\in\mathbb R^n,
\ee
where $\lambda$ is a positive constant, $\psi_2\les\psi_1$, and $H^-$ and $H^+$ denote the lower and
upper Hamiltonians, respectively, given by
\begin{equation}\begin{aligned}
H^{-}(x , y, p,X)
&:=
\sup_{u\in U}\inf_{v\in V}
\[
\frac12\operatorname{tr}\big(\sigma\sigma^\top(x,u,v)X\big)
+b(x,u,v)\cdot p
+f\big(x,y, p\sigma(x,u,v),u,v\big)
\],
\\
H^{+}(x  ,y, p,X)
&:=
\inf_{v\in V}\sup_{u\in U}
\[
\frac12\operatorname{tr}\big(\sigma\sigma^\top(x,u,v)X\big)
+b(x,u,v)\cdot p
+f\big(x,y, p\sigma(x,u,v),u,v\big)
\],
\end{aligned} \end{equation}
for $(x,y, p,X)\in\mathbb R^n\times\mathbb R\times\mathbb R^n\times\mathcal S^n$. Here
$\mathcal S^n$ denotes the set of all $n\times n$ symmetric matrices, and
$U\subset\mathbb R^{m_1}$ and $V\subset\mathbb R^{m_2}$ are compact.
The bilateral obstacle structure imposes the constraint
$\psi_2(x)\les W(x)\les \psi_1(x)$, $x\in\mathbb R^n$. In the continuation
region
$\{x\in\mathbb R^n:\psi_2(x)<W(x)<\psi_1(x)\}$, a solution $W$ satisfies
the corresponding HJBI equation, while on the contact regions
$\{x\in\mathbb R^n:W=\psi_1\}$ and $\{x\in\mathbb R^n:W=\psi_2\}$, it coincides with the upper and lower
obstacles, respectively.

  The purpose of this paper is to provide stochastic representations for the
viscosity solutions of the above stationary HJBI variational inequalities \eqref{VI-} and \eqref{VI+}. We
establish two complementary representations, both formulated on an infinite
horizon and based on BSDE-generated recursive payoffs.
 The first representation is obtained through an augmented recursive zero-sum
SDG. The basic idea is to enlarge the control spaces by adding two stopping
symbols corresponding to the two obstacles. Motivated by the enlarged-control
formulation for one-sided obstacle problems in the HJB setting
\cite{Bardi-Capuzzo-Dolcetta-1997}, we develop a two-player
bilateral-obstacle version in the recursive SDG framework with BSDE-generated
payoffs. In the augmented formulation, the obstacle values are encoded in the augmented
BSDE driver. The resulting infinite horizon recursive SDG falls within the
recursive game framework studied in
\cite{Buckdahn-Li-Zhao-2021, Huang-Wei-2026} and leads to HJBI equations without
explicit obstacle constraints. We then prove the equivalence between these HJBI
equations and the original obstacle-type HJBI variational inequalities.  Consequently, the lower
and upper value functions of the augmented  infinite horizon recursive SDG provide stochastic
representations for the bounded viscosity solutions of the original
variational inequalities.

The second representation is obtained through a mixed control-stopping SDG. In
this formulation, the players choose both  controls and stopping
decisions, and the payoff is defined by a BSDE with a random terminal time. If
the game is stopped, the terminal payoff is determined by the player who stops
first, leading naturally to the two obstacles. To connect this mixed game with
the HJBI variational inequalities, we introduce a family of penalized
infinite horizon recursive SDGs whose drivers are coupled with their own value
functions. The corresponding penalized value functions are characterized as
unique bounded viscosity solutions of stationary HJBI equations with penalized
drivers, and the penalization terms force the solutions to stay between the two
obstacles in the limit. By proving the local uniform convergence of the
penalized value functions and combining it with stability estimates for BSDEs
with random terminal times, we obtain a dynamic programming representation for
the limiting value function and identify it with the value function of the
mixed control-stopping game.

Compared with the mixed control-stopping problems and games studied in
\cite{Kamizono-Morimoto-2002,Morimoto-2003,Ghosh-Rao-2003,
Ghosh-Rao-Sheetal-2009,Ghosh-Rao-2012,Akdim-Ouknine-Turpin-2006,Hu-2020}, the
present formulation is not merely a recursive extension obtained by replacing
the standard payoff with a BSDE-generated one. A further essential feature is the
introduction of nonanticipative stopping strategies depending on the opponent's
control process. This issue does not seem to be treated as a structural part of
the strategy formulation in the above mixed-type game literature.
Such control-dependent stopping strategies are indispensable in the
Elliott--Kalton framework. Indeed, to identify the limit of the auxiliary value
functions with the value function of the original mixed control-stopping game,
one has to apply the DPP satisfied by the auxiliary value functions. The
stopping times involved in this argument are typically approximate hitting
times of the obstacle/contact regions, possibly combined with localization exit
times. Since these stopping times are defined through the controlled state
process, they necessarily depend on the opponent's control process through the
corresponding response strategy. Thus the admissible stopping strategies must
allow such nonanticipative dependence; otherwise restriction, pasting, and the
DPP would not be compatible with the strategy class.
 The stopping-strategy formulation introduced here is not an ad hoc device. In a
similar spirit, related stopping-strategy mechanisms appear in
stochastic games with stopping or impulse decisions, and in classical
differential games with optional stopping;  see, for instance,
\cite{Kalton1974,BayraktarHuang2013,Azimzadeh2019}.

 A further possible representation may be developed through infinite horizon
  reflected BSDEs with two barriers. In such an approach, the two obstacles would be imposed
directly on the backward component, and the associated value functions would
provide another probabilistic representation of the same HJBI variational
inequalities. This viewpoint is closely related to the penalization scheme used
in our second approach, since both methods enforce the bilateral obstacle
constraints through backward equations. However, a direct reflected BSDE
formulation requires a separate analysis of well-posedness, stability and
dynamic programming for infinite horizon reflected BSDEs with game features.
We therefore leave this approach for future research.

The rest of the paper is organized as follows. Section~2 introduces the
notation, assumptions, and preliminary results on BSDEs with random terminal
times. Section~3 establishes the first stochastic representation through an
augmented infinite horizon recursive SDG. Section~4 develops the mixed control-stopping SDG formulation. More precisely, it introduces a family of
penalized infinite horizon recursive SDGs, establishes the convergence of the
penalized value functions, proves the dynamic programming principle for the
penalized games, and then derives the second stochastic representation.

  \section{Preliminaries}

In this section, we introduce some notation, assumptions and preliminary definitions used
throughout the paper. We first fix the underlying stochastic basis. Let $(\Omega,\mathcal F,\mathbb F,\mathbb P)$ be a filtered probability space
equipped with a $d$-dimensional standard Brownian motion
$B=\{B_t\}_{t\ges0}$, where $\mathbb F=\{\mathcal F_t\}_{t\ges0}$ is the natural
filtration generated by $B$ and augmented by all $\mathbb P$-null sets.

For $t\ges0$, denote by $\mathcal S_t$ the collection of all
$\mathbb F$-stopping times $\tau$ satisfying
$t\les\tau\les\infty$, $\mathbb P$-a.s. In particular, we write
$\mathcal S:=\mathcal S_0$. Given $t\ges0$ and
$\tau\in\mathcal S_t$, we use
$ \ds
\llbracket t,\tau\rrbracket
:=
\big\{(s,\omega)\in[0,\infty)\times\Omega:\,
t\les s\les\tau(\omega)\big\}
$
to denote the corresponding stochastic interval. We shall work with the
following spaces.

\begin{itemize}
\item $L^1(t,\infty;\dbR)$ denotes the space of all Borel measurable functions
$\varphi:[t,\infty)\to\dbR$ such that
$
\lVert\f\rVert_{L^1}
:=
\int_t^\infty |\varphi(s)|\,\md s
<\infty .
$

\item $L^\infty_{\mathcal F_\tau}(\dbR)$ denotes the space of all
$\mathcal F_\tau$-measurable $\dbR$-valued random variables $\xi$  such that
$  \ds
\lVert\xi\rVert_\infty
:=
\esssup_{\omega\in\Omega}|\xi(\omega)|
<\infty .
$

\item $L^\infty_{\mathbb F}(t,\tau;\dbR)$ denotes the space of all
$\mathbb F$-progressively measurable $\dbR$-valued processes
$\varphi=\{\varphi_s\}_{s\ges t}$ such that
$\ds
\esssup_{(s,\omega)\in\llbracket t,\tau\rrbracket}
|\varphi_s(\omega)|<\infty .
$
Here  the essential supremum is taken with respect to $\md s\times \md\mathbb P$
on $\llbracket t,\tau\rrbracket$.

\item $\mathcal H^2_{\rm loc}(t,\tau;\dbR^d)$ denotes the space of all
$\mathbb F$-progressively measurable $\dbR^d$-valued processes
$\varphi=\{\varphi_s\}_{s\ges t}$ such that, for every $T>t$,
$
\mathbb E\int_t^{T\wedge\tau}|\varphi_s|^2\,\md s<\infty .
$

\item If $D$ is a subset of a Euclidean space, then $C_b(D)$,
$\mathrm{LSC}(D)$ and $\mathrm{USC}(D)$ stand for the spaces of bounded
continuous, lower semicontinuous and upper semicontinuous real-valued functions
on $D$, respectively.
\end{itemize}

For later use, for any probability measure $\mathbb Q$ and any $s\ges0$, we
write
$ \ds
\mathbb E_s^{\mathbb Q}[\,\cdot\,]
:=
\mathbb E^{\mathbb Q}[\, \cdot\mid\mathcal F_s].
$

\subsection{FBSDEs with random terminal times }

 We recall in this subsection some   well-posedness and stability
estimates for decoupled FBSDEs with random terminal times.
For any $(t,x)\in[0,\infty)\times\mathbb R^n$, $\tau\in\mathcal S_t$, and $\l>0$,
consider the forward SDE,
\begin{equation}\label{A.1-tau}
\left\{\2n
\begin{aligned}
&\md X_s^{t,x}  = b(X_s^{t,x})\,\md s+\sigma(X_s^{t,x})\,\md B_s,\qquad s\ges t,\\
&X_t = x,
\end{aligned}
\right.
\end{equation}
and  the BSDE
\begin{equation}
\label{A.2-tau-g}
\left\{\2n
\begin{aligned}
&Y_{s\wedge\tau}^{t,x}
 =
Y_{T\wedge\tau}^{t,x}
+
\int_{s\wedge\tau}^{T\wedge\tau}
\big(g(X_r^{t,x},\l Y_r^{t,x},Z_r^{t,x})-\l Y_r^{t,x}\big)\,\md r
-
\int_{s\wedge\tau}^{T\wedge\tau} Z_r^{t,x}\,\md B_r,
\ t\les s\les T<\infty,
\\
&Y_{\tau} ^{t,x}=h(X_\tau^{t,x}) ,\quad \text{on }\{\tau<\infty\}.
\end{aligned}
\right.
\end{equation}
On
$\{\tau=\infty\}$, we make the convention that the discounted terminal
 $
e^{-\lambda(\tau-s)}
h(X_\tau^{t,x})\mathbf 1_{\{\tau<\infty\}}
$
is zero, for every finite time $s\ges t$. Thus no value of
$X_\infty^{t,x}$ is involved.

\ss
\noindent
\textbf{Convention.}
In what follows, whenever a stopping time $\t$ appears as the terminal time of a BSDE and may take the value $\infty$, the corresponding discounted terminal term is understood to be zero on $\{\t=\infty\}$. This convention
will be used throughout the paper.

\ss
We assume that the coefficients $b: \dbR^n\to\dbR^n$, $\si:   \dbR^n\to\dbR^{n\times d},
 $ $g: \dbR^n\times \dbR\times\dbR^d\to\dbR$, and $h:  \dbR^n\to\dbR$ satisfy the following conditions.
\begin{itemize}
\item[\bf(A1)]
$b
$ and $
\sigma
$
are bounded  and Lipschitz continuous in $x\in\dbR^n$.

\item[\bf(A2)]
There exist constants $B_g,L_x,L_y,L_z\ges 0$ such that, for all  $x, x_1,x_2\in \mathbb{R}^n$, $y_1,y_2\in \mathbb{R}$, and $z_1,z_2\in \mathbb{R}^d$,
$$
|g(x_1,y_1,z_1)-g(x_2,y_2,z_2)|
\les
L_x|x_1-x_2|+L_y|y_1-y_2|+L_z|z_1-z_2|,
$$
and $|g(x,0,0)|\les B_g$.  Moreover, for all   $x\in \mathbb{R}^n$, $y_1,y_2\in \mathbb{R}$, and $z\in \mathbb{R}^d$,
$$
\bigl(g(x,y_1,z)-g(x,y_2,z)\bigr)(y_1-y_2)\les 0 .
$$

  \item  [\bf(A3)] There exist constants $B_h\ges 0$ and $L_h\ges0$    such that, for all  $x,x_1,x_2\in\mathbb R^n,$
$
|h(x)|\les B_h$, and $
|h(x_1)-h(x_2)|\les L_h|x_1-x_2|.
$

\item[{\bf(A4)}]  There exists a constant $\mu >-\l$ such that, for all $t\ges 0$ and $x,\bar x\in\mathbb R^n$,
$$
 2\langle x-\bar x,\; b(x)-b(\bar x)\rangle
+  |\sigma(x)-\sigma(\bar x) |^2+ 2L_z|\si(x)-\si(\bar x)|\cd|x-\bar x|
\les -\mu |x-\bar x|^2.
$$

\end{itemize}

 Under {\bf(A1)}, the forward SDE \eqref{A.1-tau} admits a unique strong solution.
For the BSDE \eqref{A.2-tau-g}, we shall use the following well-posedness and
stability estimates; see  \cite[Proposition 2.1, Lemmas 2.4 and 2.5 ]{Huang-Wei-2026}.

 \begin{lemma}\label{LemmaA-1-FBSDE}\sl
Under {\rm {\bf(A1)}--{\bf(A3)}}, for any $(t,x)\in [0,\infty)\times \mathbb{R}^n$  and any $\tau\in\mathcal S_t$, there exists a unique pair
$
(Y^{t,x},Z^{t,x})\in L^\infty_{\mathbb{F}}(t,\t;\mathbb{R})\times \cH^2_{\mathrm{loc}}(t,\t;\mathbb{R}^d)
$
solving \eqref{A.2-tau-g}. Moreover, for all $s\ges t$,
$\ds
 |Y_{s\wedge\tau}^{t,x}|
 \les B_h+\frac {B_g}\l,
$ $\dbP\mbox{-a.s.}
$
Moreover,  if
   {\bf(A4)} also holds, for any   $t\ges 0$,
  for all $x,\bar x\in \mathbb{R}^n$,
$$
|Y_t^{t,x}-Y_t^{t,\bar x}|
\les
\Big(
L_h^2+\frac{L_x^2}{ \l(\l+\mu)}
\Big)^{1/2}
|x-\bar x|, \qquad \mathbb P\text{-a.s.}
$$
In particular, if $h\equiv 0$, then
$\ds
|Y_t^{t,x}-Y_t^{t,\bar x}|
\les
 \frac{L_x}{\sqrt{ \l(\l+\mu)}}|x-\bar x|,$ $ \mathbb P\text{-a.s.}  $
\end{lemma}

\begin{lemma} \label{Pro-A.3}\sl  Assume   {\bf(A1)}-{\bf (A3)} hold.  For $i=1,2$, let  $(Y^i,Z^i)\in L^\infty_{\mathbb{F}}(t,\t;\mathbb{R})\times \cH^2_{\mathrm{loc}}(t,\t;\mathbb{R}^d)$   be the solutions of BSDE \eqref{A.2-tau-g} associated with $(g_i,h_i) $,  respectively, where $ g_i,h_i  $ satisfy {\bf (A2)} and {\bf (A3)} with the same constants.
Then  there exists a probability measure $\mathbb Q$, locally equivalent to
$\mathbb P$, such that, for all $s\ges t$,
$$\ba{ll}
\ns\ds
|Y_{s\wedge\tau}^1-Y_{s\wedge\tau}^2|
\les
 \mathbb E_{s\wedge\tau}^{\mathbb Q}
\[
 e^{-\lambda(\tau-s\wedge\tau)}
|h_1(X_\t^{t,x})-h_2(X_\t^{t,x})| \mathbf 1_{\{\tau<\infty\}}\\
\ns\ds \qq\qq\qq \qq\q
+
\int_{s\wedge\tau}^{\tau}
 e^{-\lambda(r-s\wedge\tau)} |g_1(X_r^{t,x},\l Y_r^2,Z_r^2)
-
g_2(X_r^{t,x},\l Y_r^2,Z_r^2)| \,\md r
\].
\ea$$

\end{lemma}

The following is the stability result of \eqref{A.2-tau-g} with respect to terminal times.

\begin{lemma}
\label{Cor-terminal-time-stability}\sl
Assume that {\bf(A1)}--{\bf(A3)} hold. Let
$\tau_1,\tau_2\in\mathcal S_t$ satisfy
$
\tau_1\les \tau_2, $ $  \mathbb P\text{-a.s.}
$
For $i=1,2$, let $(Y^i,Z^i)$ be the solution of BSDE \eqref{A.2-tau-g}  on
$\llbracket t,\tau_i\rrbracket$ with the same generator $g$ and different terminal
value
$
\xi_i\in L^\infty_{\mathcal F_{\tau_i}}(\mathbb R).
$
Then there exists a probability measure $\bar{\mathbb Q} $, locally equivalent to
$\mathbb P$, such that, for all $ s\ges t $,
$$
|Y^1_{s\wedge\tau_1}-Y^2_{s\wedge\tau_1}|
\les
\mathbb E_{s\wedge\tau_1}^{\bar{\mathbb Q}}
\left[
 e^{-\lambda(\tau_1-s\wedge\tau_1)}
|\xi_1-Y^2_{\tau_1}|\mathbf 1_{\{\tau_1<\infty\}}
\right].
$$

\end{lemma}

\begin{remark}\label{rem:measure-representation}\sl
The probability measures appearing in Lemmas~\ref{Pro-A.3} and~\ref{Cor-terminal-time-stability} can be chosen
explicitly as follows (see the proof of \cite[Proposition 2.1]{Huang-Wei-2026}).

\begin{itemize}

\item[(i)]
For the probability measure $\mathbb Q$ in Lemma~\ref{Pro-A.3}, for every $T>t$,
$
\mathbb Q|_{\mathcal F_T}=\mathbb Q^T,
 $
where
$$
\frac{\md \mathbb Q^T}{\md\mathbb P}\Big|_{\mathcal F_T}
=
\exp\Big\{
\int_t^T \zeta_r\,\md B_r
-\frac12\int_t^T |\zeta_r|^2\,\md r
\Big\},
$$
with
$$
\zeta_r:
=
1_{\{r\les\tau\}}
1_{\{Z_r^1-Z_r^2\neq0\}}
\frac{
g_1(X_r^{t,x},\lambda Y_r^2,Z_r^1)
-
g_1(X_r^{t,x},\lambda Y_r^2,Z_r^2)
}{
|Z_r^1-Z_r^2|^2
}
(Z_r^1-Z_r^2).
$$
Moreover,
$
|\zeta_\cd|\les L_z,
$ $\dbP$-a.s.

\item[(ii)]
For the probability measure $\bar{\mathbb Q}$ in Lemma \ref{Cor-terminal-time-stability}, for every
$T>t$,
 $
\bar{\mathbb Q}|_{\mathcal F_T}
=
\bar{\mathbb Q}^{\,T},
 $
where
$$
\frac{\md\bar{\mathbb Q}^{\,T}}{\md\mathbb P}\Big|_{\mathcal F_T}
=
\exp\Big\{
\int_t^T \bar\zeta_r\,\md B_r
-\frac12\int_t^T |\bar\zeta_r|^2\,\md r
\Big\},
$$
with
$$
\bar\zeta_r:
=
1_{\{r\les\tau_1\}}
1_{\{Z_r^1-Z_r^2\neq0\}}
\frac{
g(X_r^{t,x},\lambda Y_r^2,Z_r^1)
-
g(X_r^{t,x},\lambda Y_r^2,Z_r^2)
}{
|Z_r^1-Z_r^2|^2
}
(Z_r^1-Z_r^2).
$$
Moreover,
$
|\bar\zeta_r|\les L_z,
$  $\dbP$-a.s.

\end{itemize}

\end{remark}

The preceding estimates will play an important role in the analysis of the SDG considered below. We now introduce the assumptions and preliminary definitions.

\subsection{Assumptions and preliminary definitions}\label{subsec:Ass}

For the coefficients appearing in the HJBI variational inequalities
\eqref{VI-} and \eqref{VI+}, we assume that
$
b:\mathbb R^n\times U\times V\to\mathbb R^n,$
$
\sigma:\mathbb R^n\times U\times V\to\mathbb R^{n\times d},
$
$
f:\mathbb R^n\times\mathbb R\times\mathbb R^d\times U\times V\to\mathbb R,
$
$
\psi_1,\psi_2:\mathbb R^n\to\mathbb R
$
satisfy the following conditions.
\begin{description}
  \item  {\bf(C1)}   $b$ and $\sigma$ are continuous in
$(u,v) $. Moreover, they are   uniformly bounded, and Lipschitz continuous with respect to $x\in\mathbb R^n$, uniformly in
$(u,v)\in U\times V$.

  \item  {\bf(C2)} The function
$
f
$
is continuous in $(u,v)$, and  there exist constants $l_{fx}, l_{fy}, l_{fz}\ges0$ such that, for all $x_1, x_2\in \dbR^n$, $y_1, y_2\in \dbR$, $z_1, z_2\in\dbR^d$, $u\in U$ and $v\in V$,
      $$|f(x_1, y_1, z_1,u, v) - f(x_2, y_2, z_2, u, v)|\les l_{fx}|x_1-x_2| + l_{fy}|y_1-y_2|+l_{fz} |z_1-z_2| .$$
Moreover, for all   $x\in \mathbb{R}^n$, $y_1,y_2\in \mathbb{R}$, and $z\in \mathbb{R}^d$, $u\in U$ and $v\in V$,
$$
\bigl(f(x, y_1, z,u, v) - f(x,y_2, z, u, v)\bigr)(y_1-y_2)\les 0 .
$$
Finally, there  exists a constant $B_f\ges 0$ such that, for all $x\in \dbR^n$, $y\in \dbR$, $u\in U$ and $v\in V$,
 $|f(x, y, 0,u, v)  |\les B_f.$

 \item  {\bf(C3)} There is some constant $ \mu>-\l$ such that for all $x_1, x_2\in \dbR^n$, $u\in U$, $v\in V$,
$$\ba{ll}
\ns\ds  2\langle x_1-x_2,  b(x_1,u,v)-b(x_2,u, {v})\rangle + | \si(x_1,u,v)-\si(x_2,u, {v})|^2\\
\ns\ds  + 2l_{fz}| \si(x_1,u,v)-\si(x_2,u, {v})|\cd |x_1-x_2|  \les  -\m|x_1-x_2|^2 .
\ea$$

 \item  {\bf(C4)} There exist constants $B_i\ges 0$ and $l_i\ges0$, $i=1,2$, such that for all  $x,x_1,x_2\in\mathbb R^n,$
$
|\psi_i(x)|\les B_i$, and $
|\psi_i(x_1)-\psi_i(x_2)|\les l_i|x_1-x_2|.
$
Moreover,
$
\psi_2(x)\les \psi_1(x),$    for all $x\in\mathbb R^n.
$

\end{description}


We next give the definition of viscosity solutions for the lower HJBI variational inequality \eqref{VI-}. The definition for \eqref{VI+}  is completely
analogous and will therefore be omitted.

  \begin{definition}\label{VI-viscosity solution}\sl
 (i) A function $W\in \mathrm{LSC}(\mathbb R^n)$ is called a viscosity supersolution of  \eqref{VI-}, if for every test function  $\varphi\in C_{l,b}^{3}(\mathbb{R}^n)$ and $x\in\mathbb{R}^n$ at which  $W-\varphi$ attains a local minimum,  the following holds,
$$
\min\Big\{ \max\{ \lambda W (x)
-H^{-}(x, \lambda W (x), D\varphi(x),D^2\varphi(x)),W(x)-\psi_1(x)\},
W(x)-\psi_2(x)\Big\}\ges0.
$$

(ii)  A function $W\in \mathrm{USC}(\mathbb R^n)$ is called a viscosity subsolution of \eqref{VI-}, if for every  function  $\varphi\in C_{l,b}^{3}(\mathbb{R}^n)$ and $x\in\mathbb{R}^n$ at which  $W-\varphi$ attains a local maximum, the following holds,
$$
\min\Big\{ \max\{ \lambda W(x)
-H^{-}(x, \lambda W (x), D\varphi(x),D^2\varphi(x)), W(x)-\psi_1(x)\},W(x)-\psi_2(x)\Big\}\les0.
$$

(iii)  A function $W\in C(\mathbb R^n)$ is called a viscosity solution  of \eqref{VI-}, if it is both a viscosity subsolution  and a viscosity  supersolution, and satisfies $\psi_2(x)\les W(x)\les\psi_1(x)$  for all $x\in\mathbb{R}^n$.
\end{definition}

Here $C^3_{l,b}(\mathbb R^n)$ denotes the space of functions
$\varphi\in C^3(\mathbb R^n)$ such that $\varphi$ and its first-order
derivatives have at most linear growth, while all derivatives of order
two and three are bounded.

\begin{remark}\label{VI-viscosity solution-equ}\sl
 Since $\lambda>0$,    multiplying the obstacle terms
$W(x)-\psi_1(x)$ and $W(x)-\psi_2(x)$ by $\lambda$ does not change
their signs.
 Consequently,
 the inequalities in Definition \ref{VI-viscosity solution} involving
$W(x)-\psi_1(x)$ and $W(x)-\psi_2(x)$ are equivalent to those involving
$\lambda\big(W(x)-\psi_1(x)\big)$ and $\lambda\big(W(x)-\psi_2(x)\big)$, respectively.
Therefore, Definition \ref{VI-viscosity solution}  is equivalent to the standard definition of viscosity
solution  for the following   equation:
$$\begin{aligned}
\min\Big\{\!
\max\big\{\lambda W(x)-H^-(x, \lambda W (x), DW(x),D^2W(x)),\,\lambda\big(W(x)-&\psi_1(x)\big)\big\},\\
&
\,\lambda\big(W(x)-\psi_2(x)\big)
\Big\}=0,\  x\in\mathbb{R}^n.\\
\end{aligned}$$

\end{remark}

 Following the standard framework in differential games; see, for example, \cite{Buckdahn-Li-2008},
we finally introduce the admissible controls and Elliott--Kalton type
nonanticipative strategies on random intervals for later use in the study of the game problems.

\begin{definition}\label{admissible controls}\sl
Let $t\ges0$ and  $\tau\in\cS_t$.
Fix $u_0\in U$.
A process $u=\big\{u(s,\omega),(s,\omega)\in\llbracket t,\tau\rrbracket \big\}$ is called an admissible control for Player $1$ on $\llbracket t,\tau\rrbracket $, if its extension
$
u(s,\omega)\mathbf 1_{\{t\les s\les \tau(\omega)\}}
+
u_0\mathbf 1_{\{s<t\}\cup\{s>\tau(\omega)\}}
 $
is $\mathbb{F}$-progressively measurable and takes values in $U$.
The set of all admissible controls for Player 1 on $\llbracket t,\tau\rrbracket $ is denoted by $\mathcal U_{t,\tau}$.
We identify two controls $u_1$ and $u_2$ in $\mathcal U_{t,\tau}$ and write
$u_1\equiv u_2$ on $\llbracket t,\tau\rrbracket $, if
$
u_1=u_2,$ $\md s\times \md\mathbb P$-a.e. on $\llbracket t,\tau\rrbracket .
$

Similarly, the admissible controls for Player $2$ on $\llbracket t,\tau\rrbracket $ are defined by replacing
$U$, $u_0$, and $\mathcal{U}_{t,\tau}$ above with $V$, $v_0$, and   $\mathcal{V}_{t,\tau}$, respectively.

\end{definition}

\begin{definition}\sl Let $t\ges0$ and  $\tau\in\cS_t$.
A nonanticipative strategy for Player $1$ on $\llbracket t,\tau\rrbracket $ is a mapping
$
\alpha:\mathcal{V}_{t,\tau}\to\mathcal{U}_{t,\tau}
$
such that, for any $\mathbb{F}$-stopping time $S$ satisfying $t\les S\les \tau$, $\mathbb{P}$-a.s., and any $v_1,v_2\in\mathcal{V}_{t,\tau}$,
if
$
v_1\equiv v_2 \text{ on } \llbracket t,S\rrbracket ,
$
then
$
\alpha(v_1)\equiv \alpha(v_2)  \text{ on } \llbracket t,S\rrbracket .
$
The set of all nonanticipative strategies for Player $1$ on $\llbracket t,\tau\rrbracket $ is denoted by $\mathcal{A}_{t,\tau}$.

Similarly, a nonanticipative strategy for Player $2$ on $\llbracket t,\tau\rrbracket $ is a mapping
$
\beta:\mathcal{U}_{t,\tau}\to \mathcal{V}_{t,\tau},
$
defined analogously. The set of all such nonanticipative strategies is denoted by $\mathcal{B}_{t,\tau}$.

\end{definition}

 When $\tau=T$ for some deterministic $T\in[t,\infty)$ with $T<\infty$, the
above definitions reduce to the usual admissible controls and nonanticipative
strategies on the deterministic interval $[t,T]$, as in
\cite{Buckdahn-Li-2008}. When $\tau\equiv\infty$, they reduce to the
corresponding notions on the infinite horizon $[t,\infty)$; see, for example, \cite{Buckdahn-Li-Zhao-2021, Huang-Wei-2026}.

 \section{A Representation via an Augmented Infinite Horizon SDG}
\label{sec:SDG-IF}

In this section, we provide  the first probabilistic representation  of the HJBI
variational inequalities  \eqref{VI-} and \eqref{VI+} in terms of an infinite horizon SDG. The basic idea is to incorporate the bilateral obstacles into the running cost by enlarging the control spaces.

Let $\mathbf{u}\in\mathbb{R}^{m_1}$ and $\mathbf{v}\in\mathbb{R}^{m_2}$ be two additional symbols such that $\mathbf{u}\notin U$ and $\mathbf{v}\notin V$. We introduce the augmented control spaces
 $
\widetilde U:=U\cup\{\mathbf{u}\}, $ and $ \widetilde V:=V\cup\{\mathbf{v}\},
$  which are still compact.
Accordingly, one can define the corresponding sets of admissible controls $\widetilde{\mathcal U}_{0,\i}$ and $\widetilde{\mathcal V}_{0,\i}$, as well as the associated classes of nonanticipative strategies $\widetilde{\mathcal A}_{0,\i}$ and $\widetilde{\mathcal B}_{0,\i}$.

Next, we extend the coefficients $b,\sigma,f$ to
$
\widetilde b:\mathbb{R}^n\times \widetilde U\times \widetilde V\to\mathbb{R}^n,\
\widetilde \sigma:\mathbb{R}^n\times \widetilde U\times \widetilde V\to\mathbb{R}^{n\times d},\
\widetilde f:\mathbb{R}^n\times \mathbb{R}\times \mathbb{R}^d\times \widetilde U\times \widetilde V\to\mathbb{R},
$
by setting
\bel{wt-b}
\begin{aligned}&\wt{b}(x,\mathbf{u},\cd)\equiv0, ~\wt{b}(x,\cd,\mathbf{v})\equiv0,~ \wt{\sigma}(x,\mathbf{u},\cd)\equiv0,~ \wt{\sigma}(x,\cd,\mathbf{v})\equiv0,\\
&\wt f(x,y,z,u,v):=
\begin{cases}
f(x,y,z,u,v), & (u,v)\in U\times V,\\
\lambda\psi_1(x), &  u\in \widetilde{ U},\ v=\mathbf v,\\
\lambda\psi_2(x), &  u=\mathbf u,\ v\in V.
\end{cases}
\end{aligned}\ee
For any admissible controls $\widetilde u\in\widetilde{\mathcal U}_{0,\i}$ and $\widetilde v\in\widetilde{\mathcal V}_{0,\i}$,   consider the state equation
\begin{equation}\label{extend-SDE}\left\{\2n
  \begin{split}
&\md\widetilde{X}_{s}^{x;\widetilde{u},\widetilde{v}}
=\widetilde{b}(\widetilde{X}_{s}^{x;\widetilde{u},\widetilde{v}},\widetilde{u}_s,\widetilde{v}_s)\,\md s  +  \widetilde{\sigma}(\widetilde{X}_{s}^{x;\widetilde{u},\widetilde{v}}
,\widetilde{u}_s,\widetilde{v}_s)\,\md B_s, \q s\ges0,\\
&\widetilde{X}_{0}^{x;\widetilde{u},\widetilde{v}}=x, \\
\end{split}\right.
      \end{equation}
and the BSDE
      \begin{equation}\label{extend-BSDE}
      \begin{aligned}
 \widetilde Y _{s}^{x;\widetilde{u},\widetilde{v}}
=&\widetilde Y _{T}^{x;\widetilde{u},\widetilde{v}}  +  \int_{s}^{T}\big(\widetilde{f}(\widetilde{X}_{r}^{x;\widetilde{u},
\widetilde{v}}, \l\widetilde{Y}_{r}^{x;\widetilde{u},\widetilde{v}}, \widetilde{Z}_{r}^{x;\widetilde{u},\widetilde{v}},
\widetilde{u}_r,\widetilde{v}_r)  -  \lambda \widetilde Y _{r}^{x;\widetilde{u},\widetilde{v}}\big)\,\md r\\
&    -  \int_{s}^{T}\widetilde{Z}_{r}^{x;\widetilde{u},\widetilde{v}}\,\md B_r,\  0\les  s\les T<\i.\\
\end{aligned}
      \end{equation}
It is easy to verify that the extended coefficients $\widetilde b$, $\widetilde \sigma$, and $\widetilde f$
satisfy the same type of assumptions as $b$, $\sigma$, and $f$, possibly with modified constants.
More precisely, assumptions {\bf(C1)}--{\bf(C2)} remain valid, while {\bf(C3)} holds with
$\widetilde\mu:=\min\{\mu,0\}$, so that still $\widetilde\mu+\lambda>0$.
Hence, by Lemma \ref{LemmaA-1-FBSDE} with $\t\equiv\i $,
\eqref{extend-BSDE} is well posed.

We may therefore define the cost functional by
$
\widetilde J(x;\widetilde u,\widetilde v):=
\widetilde Y_0^{x;\widetilde u,\widetilde v},
$ $x\in\mathbb R^n,
$
and then introduce the lower and upper value functions by
\bel{wt-W-}
\widetilde W^-(x)
:=
\essinf_{\widetilde\beta\in\widetilde{\mathcal B}_{0,\i}}
\esssup_{\widetilde u\in\widetilde{\mathcal U}_{0,\i}}
\widetilde Y_0^{x;\widetilde u,\widetilde\beta(\widetilde u)},
\qquad x\in\mathbb R^n,
\ee
and
\bel{wt-W+}
\widetilde W^+(x)
:=
\esssup_{\widetilde \a\in\widetilde{\mathcal A}_{0,\i}}\essinf_{\widetilde v\in\widetilde{\mathcal V}_{0,\i}}
\widetilde Y_0^{x;\widetilde \a(\wt v), \widetilde v },
\qquad x\in\mathbb R^n.
\ee

\begin{theorem}\label{HJBI-VI-solution}\sl
Under  the assumptions $\mathbf{(C1)}$-$\mathbf{(C4)}$, $\widetilde W^-\in C_b(\mathbb R^n)$ and   is the unique viscosity solution of the HJBI variational inequality \eqref{VI-} in
$C_b(\mathbb R^n)$. Moreover, if $W_1\in C_b(\mathbb R^n)$ is a viscosity subsolution
and $ W_2\in C_b(\mathbb R^n)$ is a viscosity supersolution  of \eqref{VI-}, then $W_1\les W_2$ on $\mathbb{R}^n$.
\end{theorem}

\begin{proof}
By construction, the augmented controlled system introduced above is covered by
the framework of infinite horizon SDGs
developed in \cite{Huang-Wei-2026}; see also
\cite{Buckdahn-Li-Zhao-2021}. Therefore, by \cite[Corollary 3.10]{Huang-Wei-2026} or \cite[Theorem 4.1]{Buckdahn-Li-Zhao-2021}, the lower value function $\widetilde W^- $ belongs to $C_b(\mathbb R^n)$ and is the unique viscosity solution in
$C_b(\mathbb R^n)$  of the HJBI equation
 \begin{equation}\label{extend HJB equation}
 \lambda  {W}(x)-\widetilde{H} (x, \l{W}(x), D {W}(x),D^2 {W}(x))=0,\q x\in\dbR^n,
      \end{equation}
where the Hamiltonian
$\widetilde{H} :\mathbb{R}^n\times\mathbb{R}\times\mathbb{R}^n\times\dbS^n\rightarrow \mathbb{R}$ is defined by
$$\widetilde{H}  (x,y,p,A)=\sup\limits_{\widetilde{u}\in \widetilde{U}}\inf\limits_{\widetilde{v}\in \widetilde{V}}\Big\{\frac{1}{2}\tr(\widetilde{\si}\widetilde\si^\top (x,\widetilde{u},\widetilde{v})A)  +  \widetilde{b}(x,\widetilde{u},\widetilde{v})\cd p +  \widetilde{f}(x,y,p \widetilde{\sigma}(x,\widetilde{u},\widetilde{v}),\widetilde{u},\widetilde{v})\Big\}.
$$
 Moreover, the comparison principle holds for the HJBI  equation \eqref{extend HJB equation}.

By the definition \eqref{wt-b} of the extended coefficients and using
$\psi_2\les\psi_1$, one easily verifies that, for all
$(x,y,p,A)\in\mathbb{R}^n\times\mathbb{R}\times\mathbb{R}^n\times \cS^n$,
$$
\widetilde H(x,y,p,A)
=
\bigl[\,H^{-}(x,y,p,A)\wedge \lambda\psi_1(x)\,\bigr]\vee \lambda\psi_2(x).
$$
Consequently,  \eqref{extend HJB equation} is equivalent to
$$\begin{aligned}
\min\Big\{
\max\big\{\lambda W(x)-H^{-}\big(x, \l{W}(x), D  W(x),D^2  W(x)\big),
\,\l\big(  W(x)&-\psi_1(x)\big)\big\},\\
&
\,\l\big(  W(x)-\psi_2(x)\big)
\Big\}=0,\ x\in\dbR^n.\\
\end{aligned}$$
Hence, since $\lambda>0$, Remark \ref{VI-viscosity solution-equ}  implies that,
$\widetilde W^-$ is the unique viscosity solution of \eqref{VI-}.
Moreover, the comparison principle for \eqref{VI-} in $C_b(\mathbb R^n)$ follows from the corresponding one  for \eqref{extend HJB equation}.
   \end{proof}

Similar to Theorem  \ref{HJBI-VI-solution}, we have the result for the HJBI variational inequality \eqref{VI+}.

\begin{theorem}\sl
Under assumptions {\bf(C1)}--{\bf(C4)}, $\widetilde W^+\in C_b(\mathbb R^n)$ and is the unique viscosity solution of the HJBI variational inequality \eqref{VI+} in $ C_b(\mathbb R^n)$.  Moreover, the comparison  principle holds for \eqref{VI+} in  $C_b(\mathbb R^n)$.
\end{theorem}

 The proof is symmetric to that of  Theorem  \ref{HJBI-VI-solution} and is therefore omitted.

\section{A Representation via a Two-Player Zero-Sum Mixed Control--Stopping SDG}
\label{sec:SDG-MT}

In this section, we provide another probabilistic interpretation  of \eqref{VI-} and \eqref{VI+} through a mixed-type SDG involving both  controls and stopping decisions.
More precisely, the state process is governed by the following autonomous controlled SDE
\bel{state}
\begin{cases}
\md X_s = b(X_s,u_s,v_s)\,\md s+\sigma(X_s,u_s,v_s)\,\md B_s,\qquad s\ges 0,\\
X_0=x.
\end{cases}
\ee
For $\lambda>0$ and
$\theta,\tau\in\mathcal S$, consider the following BSDE:
\bel{BSDE} \left\{\2n\ba{ll}
\ns\ds
 Y_{s\wedge \theta \wedge \t}^{x;u,\th,v,\t } = Y_{ T  \wedge \theta \wedge \t}^{x;u,\th,v,\t } + \int_{s\wedge \theta \wedge \t}^{ T  \wedge\theta \wedge \t}\big( f(X_r, \l Y_r^{x;u,\th,v,\t }, Z_r^{x;u,\th,v,\t }, u_r, v_r) - \l Y_r^{x;u,\th,v,\t } \big)\,\md r \\
 \ns\ds\qq\qq\qq - \int_{s\wedge \theta \wedge \t}^{ T  \wedge\theta \wedge \t}Z_r^{x;u,\th,v,\t }\,\md B_r, \ \ \mbox{ for all }0\les s\les T<\i, \\
 \ns\ds  Y_{\theta \wedge \t}^{x;u,\th,v,\t }= \xi_{\theta,\tau},\q \mbox{on } \{\theta \wedge \t<\i \},
 \ea\right.\ee
Here,   on $\{\theta\wedge\tau<\infty\}$, $\xi_{\theta,\tau}:= {\psi}_1 (X_\t ) \mathbf 1_{\{\t \les \th\}}+  {\psi}_2(X_\th )\mathbf 1_{\{\th <\t\}}$.
If $\theta\wedge\tau=\infty$, the game is
regarded as never being stopped and the discounted terminal payoff
 is
defined to be zero. More precisely,
$
e^{-\lambda(\theta\wedge\tau)}\xi_{\theta,\tau}
=e^{-\lambda\tau}\psi_1(X_\tau)\mathbf 1_{\{\tau\les \theta,\ \tau<\infty\}}
+
e^{-\lambda\theta}\psi_2(X_\theta)\mathbf 1_{\{\theta<\tau,\ \theta<\infty\}}=0
$ on $\{\theta\wedge\tau=\infty\}.
$

 \ms

In the above setting, the coefficients $b,\sigma,f,\psi_1$, and $\psi_2$ are the same as those in Subsection  \ref{subsec:Ass} and satisfy assumptions {\bf(C1)}--{\bf(C4)}.
Under assumptions $\mathbf{(C1)}$-$\mathbf{(C3)}$,
 it follows from Lemma \ref{LemmaA-1-FBSDE} that, for any $x\in \mathbb{R}^n$, $\tau,\theta\in \cS$, and $(u,v)\in \mathcal{U}_{0,\theta\wedge\tau}\times \mathcal{V}_{0,\theta\wedge\tau}$, BSDE   \eqref{BSDE} admits a unique solution $(Y^{x;u,\th,v,\t },Z^{x;u,\th,v,\t })$.
Moreover,
$Y_{\cd}^{x;u,\th,v,\t }$ is a continuous process bounded by $B_{1}+B_{ 2}+\frac{B_f}{\lambda}$ and $Z^{x;u,\th,v,\t }\in \cH^2_{\mathrm{loc}}(0,\theta \wedge \t;\mathbb{R}^d)    $.
Then, for $x\in\dbR^n$, $\t,\th\in\cS$ and  $(u,v)\in \cU_{0,\th\wedge\t } \times  \cV_{0,\th\wedge\t }  $, we  define the cost functional by
\bel{cost}
J(x;  u,\th;v,\t ):=Y_0^{x;u,\th,v,\t }.
\ee

Based on the above construction, we now formulate a two-person zero-sum
mixed SDG involving both controls and
stopping decisions. The  controls affect the state dynamics and the
running payoff, while the stopping decisions determine the termination time
and hence the terminal payoff. More precisely, if the game is stopped at
$\theta\wedge\tau$, then the terminal payoff is $\psi_1(X_\tau)$ on
$\{\tau\les \theta\}$ and $\psi_2(X_\theta)$ on $\{\theta<\tau\}$.

For the control component, we adopt the Elliott--Kalton framework of
nonanticipative strategies against controls, see \cite{Elliott-Kalton-1972}.  By analogy with this mechanism, we
therefore formulate the stopping component as a nonanticipative rule depending
on the opponent's control process. More precisely, the stopping time chosen by
a player may depend on the opponent's control path only through its past and
present values, but not through its future values. This leads to the following
classes of nonanticipative stopping strategies.

\begin{definition}\label{Def-ST}\sl
Let $t\ges 0$.
A mapping $\boldsymbol{\theta}:\mathcal V_{t,\infty}\to\mathcal S_t$ is called a
nonanticipative stopping strategy for Player 1 on $[t,\infty)$ if, for any
$\mathbb F$-stopping time $\varrho$ with $\varrho\ges t$ and any
$v_1,v_2\in\mathcal V_{t,\infty}$,
if $v_1 \equiv v_2$ on $\llbracket t,\varrho\rrbracket$, then
$
\boldsymbol{\theta}[v_1]\wedge \varrho
=
\boldsymbol{\theta}[v_2]\wedge \varrho,
$ $ \mathbb P\hbox{-a.s.}
$
We denote by $\mathfrak S_{t,\infty}$ the set of all
nonanticipative stopping strategies for Player 1 on $[t,\infty)$.

Similarly, a mapping $\boldsymbol{\tau}:\mathcal U_{t,\infty}\to\mathcal S_t$
is called a nonanticipative stopping strategy for Player 2 on $[t,\infty)$ if
the same nonanticipativity condition holds with $v_1,v_2$ replaced by
$u_1,u_2\in\mathcal U_{t,\infty}$. The set of all such strategies is denoted
by $\mathfrak T_{t,\infty}$.
\end{definition}

This formulation is consistent with the notion of strongly nonanticipative
families of stopping times used in SDGs with stopping or impulse decisions;
see, for example, \cite{BayraktarHuang2013, Azimzadeh2019}. A deterministic
analogue already appears in optional stopping games, where stopping rules are
formulated as nonanticipative mappings of the players' control paths; see
Kalton \cite{Kalton1974}. A further explanation of this choice will be given
in Remark \ref{Remark-nonanticipative-stopping-strategy} below, where we point
out the difficulty that arises if the stopping component is formulated only in
terms of ordinary stopping times.

 In the lower game, Player 2, the minimizer, first chooses a pair
$(\beta,\boldsymbol{\tau})\in
\mathcal B_{0,\infty}\times\mathfrak T_{0,\infty}$,
  both acting against
the control of Player 1. After this response rule is fixed, Player 1 chooses
a control $u\in\mathcal U_{0,\infty}$ and a stopping time
$\theta\in\mathcal S$. The realized control and stopping time of Player 2 are
then given by $\beta(u)$ and $\boldsymbol{\tau}[u]$, respectively.
The upper game is formulated symmetrically, with the roles of the two players
interchanged. Accordingly, we define the lower and upper value functions,
respectively, by
\bel{LU-VF}\ba{ll}
\ns\ds W^-(x)
:=
\essinf_{\substack{\beta\in\mathcal B_{0,\infty}\\
\boldsymbol{\tau}\in\mathfrak T_{0,\infty}}}
\esssup_{\substack{u\in\mathcal U_{0,\infty}\\
\theta\in\mathcal S}}
J\bigl(x;u,\theta;\beta(u),\boldsymbol{\tau}[u]\bigr),\\
\ns\ds W^+(x)
:=
\esssup_{\substack{\alpha\in\mathcal A_{0,\infty}\\
\boldsymbol{\theta}\in\mathfrak S_{0,\infty}}}
\essinf_{\substack{v\in\mathcal V_{0,\infty}\\
\tau\in\mathcal S}}
J\bigl(x;\alpha(v),\boldsymbol{\theta}[v];v,\tau\bigr).
\ea\ee
%

 %
%
%
%
%
%
Before stating the main result of this section, we impose the following compatibility condition.
\begin{description}

  \item  {\bf(C5)} The constant $l_{fy}$ satisfies $0\les l_{fy}<1$, and, whenever $\mu\les0$, we require $0\les l_{fy} < \sqrt{1 + \frac{\mu}{\lambda}}$.

\end{description}

\begin{theorem}\label{original problem value function existence}\sl
Assume {\bf(C1)}--{\bf(C5)}. Then the lower and upper value functions $W^{-}$ and $W^{+}$ of the mixed-type SDG  coincide with the unique viscosity solutions of the variational inequalities \eqref{VI-} and \eqref{VI+}, respectively.
\end{theorem}

%

In the following, only the statement for the lower value function $W^-$ will be proved; the argument for the upper value function $W^+$ is completely analogous.
Before proceeding, we briefly explain the main idea of the proof.
First, we introduce a family of penalized
SDGs and show that their value functions
$W_\varepsilon$ converge to the unique viscosity solution $W$ of
\eqref{VI-}.  Then, using the DPP established for $W_{\varepsilon}$, we derive, through a limiting argument, a corresponding dynamic programming representation for $W$. This representation then
allows one to return to the original mixed-type differential game and identify $W$ with the lower value function $W^-$.

\subsection{Penalized infinite horizon SDGs coupled with their value functions}

For any $\varepsilon>0$,  we set $l_\varepsilon:=\frac{\e l_{fx}+l_1+l_2}{ \sqrt{(\l \e+1)((\l+\m)\e+1)}-(l_{fy}\l\e + 1)}>0$, $B_\e:=\frac{\e B_f+B_1+B_2}{\l\e}$ and  define
$$
\mathrm{Lip}_{\e}^b(\mathbb{R}^n)
:=
\big\{\f\in C_b(\dbR^n): |\varphi(x) |\les B_\e,\    |\varphi(x)-\varphi(y)|\les l_\varepsilon |x-y|,
\ \forall x,y\in\mathbb R^n \big \}.
$$
Equipped with the metric induced by the uniform norm
$ \ds
\lVert\varphi\rVert_{\infty ;\dbR^n}
:=
\sup_{x\in\mathbb R^n}|\varphi(x)|,
$
$\mathrm{Lip}_{\e}^b(\mathbb R^n)$ is a complete metric space.
For later use, for  $\f\in\mathrm{Lip}_{\e}^b(\mathbb{R}^n),$ $  (x,z)\in\dbR^n\times\dbR^d,$ $(u,v)\in U\times V$, we also  set
\bel{}\ba{ll}
\ns\ds P_\e^\f(x ):= \frac{1}{\e}\f(x)
-\frac{1}{\e}\big(\f(x)-\psi_1(x)\big)^{+}
+\frac{1}{\e}\big(\f(x)-\psi_2(x)\big)^{-} , \\
\ns\ds F_\e^{\f,\l}(x,z, u,v):= f(x,\l\f,z,u,v)+P_\e^\f(x ).\\
\ea
\ee
Using $\psi_2\les \psi_1$, we observe that
$ \ds
P_{\varepsilon}^{\varphi}(x)
=
\frac1\varepsilon
\bigl((\varphi(x)\vee \psi_2(x))\wedge \psi_1(x)\bigr),$ $x\in\dbR^n.
$
Then, for any $x, \bar{x}\in \dbR^n$ and $\f , \f_ 1, \f_2\in \mathrm{Lip}_{\e}^b(\mathbb{R}^n) $, we have
\begin{equation}\label{positive-negative-part-estimate-1}
\ba{ll}
\ns\ds{\rm (i) }\  |P_\e^{\f}(x)  |\les\frac{ B_\e+B_{1}+B_{2}}{\e} , \\
\ns\ds{\rm (ii) }\  |P_\e^{\f_1}(x)-P_\e^{\f_2}(x)-\frac{1}{\e}(\f_1(x)-\f_2(x)) |\les\frac{1}{\e} |\f_1(x)-\f_2(x)|, \\
\ns\ds {\rm (iii) }\ \big| P_\e^{\f_1}(x )-P_\e^{\f_2}(x )\big|\les\frac{1}{\e}|\f_1(x)-\f_2(x)|, \\
\ns\ds  {\rm (iv) }\ \big| P_\e^{\f }(x )-P_\e^{\f }(\bar x )\big| \les\frac{1}{\e}\big(|\f(x)-\f(\bar x)| +|\psi_1(x)-\psi_1(\bar{x})|
+|\psi_2(x)-\psi_2(\bar{x})|\big).
\ea
\end{equation}
Therefore, for any $x, \bar{x}\in \dbR^n$, $z, \bar{z}\in \dbR^d$   and $\f , \f_ 1, \f_2\in \mathrm{Lip}_{\e}^b(\mathbb{R}^n) $,
 \begin{equation}\label{Con-F}
\ba{ll}
\ns\ds{\rm (i) }\  |F_\e^{\f,\l}(x,z, u,v)-F_\e^{\f,\l}(\bar x,\bar z, u,v) |\les\(l_{fx}+ l_{fy}\l l_{\e}  +\frac{l_{\e}+l_{1}+l_{2}}{\e} \)
|x-\bar x|+l_{fz}|z-\bar z|, \\
\ns\ds{\rm (ii) }\  |F_\e^{\f_1,\l}(x,z, u,v)-F_\e^{\f_2,\l}(x,z, u,v) |\les(\frac{1}{\e} + l_{fy}\l)|\f_1(x)-\f_2(x)| , \\
\ns\ds {\rm (iii) }\  | F_\e^{\f,\l}(x,0,u,v)|\les B_f+ \frac{ B_\e+B_{1}+B_{2}}{\e}   .
%
\ea
\end{equation}

For any $\varepsilon>0$, we consider the following penalized system:
for each $x\in\mathbb R^n$ and each admissible control pair
$(u,v)\in\mathcal U_{0,\infty}\times\mathcal V_{0,\infty}$,
 \bel{penalized BSDE}\left\{\2n\ba{ll}
 \ns\ds Y_{s}^{\e,x;u,v} = Y_{ T}^{\e,x;u,v}+\int_{s  }^{ T }\( F_\e^{W_\e,\l}(X_r^{x;u,v}, Z_r^{\e,x;u,v} ,u_r, v_r)-\big(\l+\frac1\e \big) Y_r ^{\e,x;u,v}\)\md r \\
\ns\ds  \hskip 2cm   - \int_{s  }^{ T }Z_r^{\e,x;u,v}\md B_r,\q \forall 0\les s\les T<\i ,\\
 \ns\ds  W_\e(x)=\essinf_{\b \in \cB_{0,\i} } \esssup_{u \in \cU_{0,\i}  }   Y_{0}^{\e,x;u,\b(u)} ,\q x\in\dbR^n,
\ea\right.\ee
 where $X^{x;u,v}_\cd$ satisfies \eqref{state}.
Unlike a standard BSDE, \eqref{penalized BSDE} is coupled with its lower
value function $W_\varepsilon$ through the generator
$F_\varepsilon^{W_\varepsilon}$. Similar coupled systems also appear in
\cite{Hao-Li-2017, Li-Li-2015, Yang-Zhang-Zhou-2023}. We first establish the
following result for \eqref{penalized BSDE}.

\begin{proposition}\label{lower value function existence}\sl
Under  assumptions {\bf(C1)}-{\bf(C5)}, for any $x\in\dbR^n$, $(u(\cd),v(\cd))\in   \cU_{0,\i}\times \cV_{0,\i}$, the penalized system \eqref{penalized BSDE} admits a unique solution $ (Y^{\e,  x;u,v},Z^{\e,x;u,v},W_\e)\in L^{\i}_{\mathbb{F}}(0,\i;\dbR)\times \cH^2_{\mathrm{loc}}(0,\i;\mathbb{R}^d)\times \mathrm{Lip}_{\e}^b(\mathbb{R}^n)$.
Moreover, the following property holds:

(i) for all $t\ges 0$, $x,\bar x\in\dbR^n$, and $u  \in\cU_{0,\i}$, $v  \in\cV_{0,\i}$,
$$|Y_{t}^{\e,x;u,v}|\les B_\e ,
\q |W_{\e}(x)  |\les B_\e ,\q
 |W_{\e}(x) -W_{\e}(\bar x)|\les  l_\varepsilon|x-\bar x|.$$

(ii)  $W_{\e}(\cd)$ is deterministic and is  the unique viscosity solution in $C_b(\dbR^n)$ of the following HJBI equation
 \begin{equation}\label{HJB equation-penalized value function}
  \begin{aligned}
&\lambda W_{\e}(x)+\frac{1}{\e}\big(W_{\e}(x)-\psi_1(x)\big)^{+}
-\frac{1}{\e}\big(W_{\e}(x)-\psi_2(x)\big)^{-}\\
&\qq\qq  -H^{-}\big(x,\l W_{\e}(x),DW_{\e}(x),D^2W_{\e}(x)\big)=0,\ x\in\dbR^n.
  \end{aligned}
      \end{equation}


\end{proposition}

\begin{proof}

For any  ${\bf w} \in \mathrm{Lip}_{\e}^b(\mathbb{R}^n) $,  we first consider an
auxiliary game whose lower value function is given by
\begin{equation}\label{auxiliary SCP}
W_{\e}^{{\bf w}}(x)=\essinf\limits_{\b\in \cB_{0,\i}}\esssup\limits_{u\in \cU_{0,\i}}Y_{0}^{\e,{\bf w},x;u,\b(u)}, \q x\in\dbR^n,
\end{equation}
where
 \begin{equation}\label{pBSDE-1}
\begin{aligned}
 Y_{s}^{\e,{\bf w},x;u,v}
&=Y_{T}^{\e,{\bf w},x;u,v}  +  \int_{s}^{T}\(F_\e^{ {\bf w},\l} (X_{r}^{x;u,v},
Z_{r}^{\e,{\bf w},x;u,v},  u_r,v_r)  -  (\lambda+\frac{1}{\e}) Y_{r}^{\e,{\bf w},x;u,v}\)\md r\\
&\qq-\int_{s}^{T}Z_{r}^{\e,{\bf w},x;u,v}\md B_r, \q \forall 0\les s\les T<\i.
\end{aligned}
      \end{equation}
Note that   \eqref{Con-F}, by  Lemma   \ref{LemmaA-1-FBSDE}  with $\t\equiv\i$ and the constant discount
$
 \lambda+\frac{1}{\varepsilon},
$   BSDE \eqref{pBSDE-1} admits a unique solution $(Y^{\e,{\bf w},x;u,v},Z^{\e,{\bf w},x;u,v}) \in L_\dbF^\i(0,\i;\dbR)\times\cH^2_{\mathrm{loc}}(0,\i;\mathbb{R}^d)$, and
$$\ba{ll}
\ns\ds
|Y_{0}^{\e,{\bf w},x;u,v}|\les \frac{ \e  }{\lambda\e+1} \(B_f+ \frac{ B_\e+B_{1}+B_{2}}{\e} \) \les B_\e ,\q \dbP\mbox{-a.s.},\\
\ns\ds |Y_{0}^{\e,{\bf w},x;u,v}-Y_{0}^{\e,{\bf w},\bar x;u,v}|\les \frac{ \e}{\sqrt{(\lambda\e+1)((\lambda+\mu)\e+1)}}\( l_{fx}+ l_{fy}\l l_{\e}  + \frac{l_\varepsilon+l_{1}+l_{2}}{\e}\)|x-\bar x|\\
\ns\ds
\qq\qq\qq\qq\qq
\les l_\varepsilon|x-\bar x|,\  \dbP\mbox{-a.s.},\ea$$
  for all $s\ges 0$, $x,\bar x\in\dbR^n$, and    $u  \in\cU_{0,\i}$, $v  \in\cV_{0,\i}$.
Then, for all   $x,\bar x\in\dbR^n$, we    have
\bel{W-w-Pro}
|W_{\e}^{{\bf w}}(x)  |\les  B_\e,\qq
|W_{\e}^{{\bf w}}(x) -W_{\e}^{{\bf w}}(\bar x)|\les l_\varepsilon|x-\bar x|.
\ee
Furthermore,  by \cite[Theorem 4.1]{Buckdahn-Li-Zhao-2021} or \cite[Corollary 3.10]{Huang-Wei-2026}, $W_{\e}^{{\bf w}} $ in \eqref{auxiliary SCP} is deterministic and is the unique viscosity solution  in $\mathrm{Lip}_{\e}^b(\mathbb{R}^n) $ of the following HJBI equation,
 \begin{equation}\label{HJBI-au}
 \ba{ll}
 \ns\ds
\big(\lambda +\frac 1 \e  \big)W_{\e}^{{\bf w}}(x) -\sup_{u  \in U }\inf_{v \in V}  \big[ \cL^{u,v}W_{\e}^{{\bf w}}(x)    + F_\e^{ {\bf w},\l} (x,   DW_{\e}^{{\bf w}}(x) \si(x,u,v) ,u, v) \big] =0,\q x\in\dbR^n,
   \ea   \end{equation}
 where, for $(u,v)\in U\times V$, $\ds\mathcal L^{u,v}\phi(x):=\frac12\operatorname{tr}\bigl(\sigma\sigma^\top(x,u,v)D^2\phi(x)\bigr)+b(x,u,v)\cdot D\phi(x).$

With the above preparation, \eqref{pBSDE-1} and  \eqref{auxiliary SCP}  define a mapping
$$\cT: \mathrm{Lip}_{\e}^b(\mathbb{R}^n) \to \mathrm{Lip}_{\e}^b(\mathbb{R}^n) \mbox{ with }\cT({\bf w})= W_{\e}^{{\bf w}}.$$
Next, we show that $\cT $ is  contractive. To this end,
for   ${\bf w}_1, {\bf w}_2\in \mathrm{Lip}_{\e}^b(\mathbb{R}^n) $, we set   $ (\widehat Y  ,\widehat Z) :=(Y^{\e,{\bf w}_1,x;u,v}-Y^{\e,{\bf w}_2,x;u,v},Z^{\e,{\bf w}_1,x;u,v}-Z^{\e,{\bf w}_2,x;u,v}) $. Then $(\widehat Y,\widehat Z)$ satisfies, for $0\les s\les T<\infty$,
\bel{widehat Y-1}
\begin{aligned}
\widehat Y_s
&=
\widehat Y_T
+
\int_s^T
\Big[
-\(\lambda+\frac1\varepsilon\)\widehat Y_r
+
F_\varepsilon^{\mathbf w_1,\l}
 \bigl(X_r^{x;u,v},
       Z_r^{\varepsilon,\mathbf w_1,x;u,v},
       u_r,v_r\bigr)
\\
&\qquad\qquad\qquad
-
F_\varepsilon^{\mathbf w_2,\l}
 \bigl(X_r^{x;u,v},
       Z_r^{\varepsilon,\mathbf w_2,x;u,v},
       u_r,v_r\bigr)
\Big]\,\md r
-
\int_s^T \widehat Z_r\,\md B_r .
\end{aligned}
\ee
To linearize the difference with respect to the $z$-variable, we define  the  process $\eta$ by
$$
\eta_r
:=
\begin{cases}
\dfrac{
F_\varepsilon^{\mathbf w_1,\l}
 \bigl(X_r^{x;u,v},
       Z_r^{\varepsilon,\mathbf w_1,x;u,v},
       u_r,v_r\bigr)
-
F_\varepsilon^{\mathbf w_1,\l}
 \bigl(X_r^{x;u,v},
       Z_r^{\varepsilon,\mathbf w_2,x;u,v},
       u_r,v_r\bigr)}
{|\widehat Z_r|^2}\,\widehat Z_r,
& \widehat Z_r\ne 0, \\[2.2ex]
0, & \widehat Z_r=0 .
\end{cases}
$$
By  {\bf(C2)}, we know $\eta$ is bounded. Then
$\ds
\mathcal E_T
:=
\exp\Big\{
\int_0^T \eta_r\,\md B_r
-
\frac12\int_0^T |\eta_r|^2\,\md r
\Big\}
$
is a martingale, so that we  can define an equivalent probability measure $\widetilde{\mathbb P}$
on $\mathcal F_T$ by
$ \ds
\frac{\md\widetilde{\mathbb P}}{\md\mathbb P}
:=
\mathcal E_T .
$
By Girsanov's theorem,
$ \ds
\widetilde B_s
:=
B_s-\int_0^s \eta_r\,\md r,
$ $ 0\les s\les T,
$
is a Brownian motion under $\widetilde{\mathbb P}$.
Consequently, under
$\widetilde{\mathbb P}$, we obtain
 \begin{equation}\label{pBSDE-2}
\widehat Y_s=\widehat Y_T+\int_s^T\(-\big(\lambda+\frac{1}{\varepsilon}\big)\widehat Y_r+\Delta F_r\)\,\md r-\int_s^T \widehat Z_r\,\md\widetilde B_r,
\qquad 0\les s\les T<\infty,
     \end{equation}
where
%
$ \D F_r:=F_\e^{ {\bf w}_1,\l} (X_{r}^{x;u,v},
Z_{r}^{\e,{\bf w}_2,x;u,v},  u_r,v_r)-F_\e^{ {\bf w}_2,\l} (X_{r}^{x;u,v},
Z_{r}^{\e,{\bf w}_2,x;u,v},  u_r,v_r).$
%

Applying the It\^o's formula to $e^{-(\lambda+\frac 1 \varepsilon ) s}\,\widehat Y_s$ yields
 \begin{equation}\label{contraction inequality}
\begin{aligned}
\ba{ll}
\ns\ds |  \widehat Y_{0}|= \Big|\dbE^{\widetilde{\mathbb P}}\[ e^{-(\lambda+\frac{1}{\e}) T }\widehat Y_{T} +\int_0^T e^{-(\lambda+\frac{1}{\e}) s}  \D F_s\, \md s\]\Big|\\
\ns\ds
\q\  \les \dbE^{\widetilde{\mathbb P}}\[ e^{-(\lambda+\frac{1}{\e}) T }|\widehat Y_{T}|  +  \int_0^T e^{-(\lambda+\frac{1}{\e}) s}(\frac{1}{\e} + l_{fy}\l)|{\bf w}_1(X_{s}^{x;u,v})-{\bf w}_2(X_{s}^{x;u,v})|\,\md s\]\\
\ns\ds
\q\  \les  2B_\e  e^{-(\lambda+\frac{1}{\e}) T }  +  \frac{l_{fy}\l\e + 1}{\l \e+1} \Arrowvert{\bf w}_1-{\bf w}_2\Arrowvert_{\infty;\dbR^n}.
\ea
\end{aligned}
      \end{equation}
Letting $T\to\i$ and using the definitions of essential  supremum and infimum, we obtain
$$\Arrowvert  \cT({\bf w_1})-\cT({\bf w_2}) \Arrowvert_{\infty;\dbR^n} =\Arrowvert   W_{\e}^{{\bf w}_1}-W_{\e}^{{\bf w}_2}\Arrowvert_{\infty;\dbR^n}\les \frac{l_{fy}\l\e + 1}{\l \e+1}  \Arrowvert  {\bf w}_1-{\bf w}_2\Arrowvert_{\infty;\dbR^n},  $$
which, along with {\bf(C5)}, yields  the unique ${\bf w}^* \in \mathrm{Lip}_{\e}^b(\mathbb{R}^n) $ such that $ \cT({\bf w}^*)=W_\e^{{\bf w}^\ast}={\bf w}^*$. In particular, since the fixed point
is obtained in the deterministic function space $ \mathrm{Lip}_{\e}^b(\mathbb{R}^n) $,
the function ${\bf w}^*$ is deterministic.
For this choice of ${\bf w}^*$, the auxiliary
BSDE \eqref{pBSDE-1} is exactly the BSDE in the penalized system \eqref{penalized BSDE}, with
$W_\varepsilon$ replaced by ${\bf w}^*$. So, $
W_\varepsilon=W_\e^{{\bf w}^\ast}={\bf w}^* .
$
Then, $(Y^{\e,{\bf w}^*,x;u,v},$ $Z^{\e,{\bf w}^*,x;u,v}, {\bf w}^* )\in L^{\i}_{\mathbb{F}}(0,\i;\dbR)\times \cH^2_{\mathrm{loc}}(0,\i;\mathbb{R}^d)\times \mathrm{Lip}_{\e}^b(\mathbb{R}^n)$ is the unique solution of the system \eqref{penalized BSDE}. Moreover, property (i) follows from \eqref{W-w-Pro}.

 Since $W_\varepsilon=W_\varepsilon^{\mathbf w^*}$, taking
$\mathbf w^*=W_\varepsilon$ in \eqref{HJBI-au}, together with the
definition of $P_\varepsilon^{W_\varepsilon}$, shows that the viscosity
subsolution and supersolution inequalities for \eqref{HJBI-au} coincide
exactly with those for
\eqref{HJB equation-penalized value function}. Hence
$W_\varepsilon$ is a viscosity solution of
\eqref{HJB equation-penalized value function}.

 Finally, for $(x,y,z,u,v)\in
\mathbb R^n\times\mathbb R\times\mathbb R^d\times U\times V$, define
$$
g(x,y,z,u,v)
:=
f(x,y,z,u,v)
-\frac1\varepsilon\left(\frac{y}{\lambda}-\psi_1(x)\right)^+
+\frac1\varepsilon\left(\frac{y}{\lambda}-\psi_2(x)\right)^- .
$$
Then, under assumptions {\bf(C1)}--{\bf(C5)}, the generator $g$ satisfies all
the conditions required in \cite[Corollary 3.10]{Huang-Wei-2026}. With this
choice of $g$ and with the constant discount coefficient $\rho\equiv\lambda$,
the corresponding HJBI equation becomes exactly
\eqref{HJB equation-penalized value function}. Hence the comparison principle
in \cite[Corollary 3.10]{Huang-Wei-2026} implies the uniqueness of bounded
viscosity solutions to \eqref{HJB equation-penalized value function}.
 \end{proof}

\subsection{Convergence of the penalized  value functions}\label{ }

In this subsection, we study the convergence of $W_\varepsilon$ as
$\varepsilon\to0$.

\begin{lemma}\label{penalized value function-uniformly boundness}\sl
Assume {\bf(C1)}-{\bf(C5)} hold true. Then, for all $x\in\dbR^n$, we have
\begin{equation}\label{penalized value function-uniformly boundness-formula}
\mathop{\limsup}\limits_{\e\rightarrow0}W_{\e}(x)\les\psi_1(x)~\mathrm{ and }~\psi_2(x)  \les  \mathop{\liminf}\limits_{\e\rightarrow0}  W_{\e}(x)
.
\end{equation}
\end{lemma}

\begin{proof}
For $\delta>0$, let   ${\widetilde\psi}_1^{\delta}$ be  the standard mollification of $\psi_1$, then define
$$
\psi_1^\delta(x)
:=
\widetilde\psi_1^\delta(x)+\lVert\widetilde\psi_1^\delta-\psi_1\rVert_{\infty;\dbR^n},\q x\in\dbR^n.
$$
Then $ \psi_1^\delta\to\psi_1$ uniformly as $\delta\to0$, and
$
\psi_2(x)\les \psi_1(x)\les  \psi_1^\delta(x), $ $ x\in\dbR^n.
$
Moreover, by {\bf(C4)},  $\psi_1^\d$ is  Lipschitz continuous and
 uniformly bounded with $B_1+\lVert\widetilde\psi_1^\delta-\psi_1\rVert_{\infty;\dbR^n}$.
We also note that the bounds for $D\psi_1^\delta$ and $D^2\psi_1^\delta$ generally depend on $\delta$.

 Applying It$\mathrm{\hat{o}}$'s formula to $\psi_1^{\delta}(X_{s}^{x;u,v})$, we obtain
 \begin{equation*}\label{psi1-equation}
 \begin{aligned}
&\psi_1^{\delta}(X_{s}^{x;u,v})
=\psi_1^{\delta}(X_{T}^{x;u,v}) -\int_{s}^{T} \cL^{u,v} \psi_1^{\delta}(X_{r}^{x;u,v})\, \md r\\
&\qq\qq\qq -\int_{s}^{T} D\psi_1^{\delta}(X_{r}^{x;u,v})  \sigma(X_{ r}^{x;u,v},u_r,v_r)\,\md B_r,\q \forall 0\les s\les T<\i.
\end{aligned}
      \end{equation*}
      %
Setting
$ Y^{1,\e }_s:=Y_{s}^{\e,x;u,v}-  \psi_1^{\delta}(X_{s}^{x;u,v}) , $ $ Z^{1,\e }_s:=Z_{s}^{\e,x;u,v}- D\psi_1^{\delta}(X_{s}^{x;u,v}) \si(X_{s}^{x;u,v},u_s,v_s)    $
 and using  the BSDE in \eqref{penalized BSDE},  we know
    \begin{equation*}\label{Y-psi1-difference-version 1}
\begin{aligned}
 Y_s^{1,\e}
 =&
 Y_T^{1,\e}
+\int_s^T
\Big(
f\bigl(
X_r^{x;u,v}, \lambda W_\varepsilon(X_{r}^{x;u,v}),
Z_r^{1,\e}
+
D\psi_1^\delta(X_r^{x;u,v})
\sigma(X_r^{x;u,v},u_r,v_r),
u_r,v_r
\bigr) \\
&
+\mathcal L^{u,v}\psi_1^\delta(X_r^{x;u,v})
+P_{\varepsilon}^{W_\varepsilon}(X_r^{x;u,v})
-(\l+\frac 1\e)
\psi_1^\delta(X_r^{x;u,v})
-
(\lambda+\frac 1\e) Y_r^{1,\e}
\Big)\md r\\
&
-\int_s^T
 Z_r^{1,\e}\,\md B_r,\  \forall 0\les s\les T<\i .
\end{aligned}
\end{equation*}

 By the linearization argument with respect to $z$ and the associated
Girsanov transform, there exists an equivalent probability measure
$\mathbb P^1$ and a $\mathbb P^1$-Brownian motion $B^1$ such that
$$
\begin{aligned}
&Y_s^{1,\e}
=
Y_T^{1,\e}
+\int_s^T
\Big(
f\bigl(X_r^{x;u,v}, \lambda W_\varepsilon(X_{r}^{x;u,v}),
D\psi_1^\delta(X_r^{x;u,v})\sigma(X_r^{x;u,v},u_r,v_r),
u_r,v_r\bigr)
+\mathcal L^{u,v}\psi_1^\delta(X_r^{x;u,v})\\
&\quad
+P_{\varepsilon}^{W_\varepsilon}(X_r^{x;u,v})
-(\l+\frac1\e)  \psi_1^\delta(X_r^{x;u,v})
-(\lambda+\frac1\e) Y_r^{1,\e}
\Big)\md r
-\int_s^T  Z_r^{1,\e}\,\md B_r^1, \  \forall 0\les s\les T<\i .
\end{aligned}
$$
Setting $\ds
\Theta_{0,s}
:=
\exp\Big\{ -\(\lambda+\frac1\e\)s\Big\},
$
and applying It\^o's formula to
$\Th_{0,s} \, Y_s^{1,\e}$ gives
$$
\begin{aligned}
  Y_0^{1,\e}
&=
\mathbb E^1\[
\Th_{0,T}   Y_T^{1,\e}
+\int_0^T \Th_{0,r}
\Big(
f\bigl(X_r^{x;u,v}, \lambda W_\varepsilon(X_{r}^{x;u,v}),
D\psi_1^\delta(X_r^{x;u,v})\sigma(X_r^{x;u,v},u_r,v_r),
u_r,v_r\bigr))\\
&\qquad\qquad \qq\qq\qq\qq
 +\mathcal L^{u,v}\psi_1^\delta(X_r^{x;u,v} + P_{\varepsilon}^{W_\varepsilon}(X_r^{x;u,v})
-(\l+\frac1\e)
 \psi_1^\delta(X_r^{x;u,v})
\Big)\md r
\]\\
&\les
\mathbb E^1\[
\Th_{0,T}   Y_T^{1, \e}
+\int_0^T \Th_{0,r}
\Big(
f\bigl(X_r^{x;u,v}, \lambda W_\varepsilon(X_{r}^{x;u,v}),
D\psi_1^\delta(X_r^{x;u,v})\sigma(X_r^{x;u,v},u_r,v_r),
u_r,v_r\bigr) \\
&\qquad\qquad\qq\qq\qq\qq
+\mathcal L^{u,v}\psi_1^\delta(X_r^{x;u,v})
-\lambda\psi_1^\delta(X_r^{x;u,v})
\Big)\md r
\].
\end{aligned}
$$
In the above, we have used
  $\ds
P_{\varepsilon}^{W_\varepsilon}(x)
\les
\frac1\varepsilon
 \psi_1^\delta(x)
 ,$ $ x\in\mathbb R^n,
$
which is from $\psi_2\les \psi_1\les \psi_1^\delta$.

By assumptions {\bf(C1)}, {\bf(C2)},  {\bf(C4)}, there exists a constant
 $C_\delta>0$, independent of $\varepsilon$, $T$, $u$ and $v$,
such that
$
\big|
f\bigl(x,\lambda W_\varepsilon(x),D\psi_1^\delta(x)\sigma(x,u,v),u,v\bigr)
+\mathcal L^{u,v}\psi_1^\delta(x)
-\lambda\psi_1^\delta(x)
\big|
\les C_\delta .
$
Hence
$$
Y_0^{1, \e}
\les
\big(B_\e+B_1+\lVert\widetilde\psi_1^\delta-\psi_1\rVert_{\infty;\dbR^n}\big)e^{-(\lambda+\frac1\varepsilon)T}
+
C_\delta\int_0^T e^{-(\lambda+\frac1\varepsilon)r}\,\md r .
$$
Letting $T\to\infty$, we obtain
$\ds
 Y_0^{1, \e}
\les
\frac{\varepsilon C_\delta}{\lambda\varepsilon+1}.
$
Then, for $x\in\dbR^n$, we have
\begin{equation*}\label{W-e-psi-1}
\begin{aligned}
W_\varepsilon(x)-\psi_1^\delta(x)
&=
\essinf_{\beta\in\mathcal B_{0,\infty}}
\esssup_{u\in\mathcal U_{0,\infty}}
\big(
 Y_0^{ \e,x;u,\beta(u)}
-\psi_1^\delta(x)
\big)  =
\essinf_{\beta\in\mathcal B_{0,\infty}}
\esssup_{u\in\mathcal U_{0,\infty}}
 Y_0^{1, \e}
\les
\frac{\varepsilon C_\delta}{\lambda\varepsilon+1}.
\end{aligned}
\end{equation*}
Letting $\e\to0$ and then $\delta\to0$, we obtain
 $\mathop{\limsup}\limits_{\e\rightarrow0}W_{\e}(x)\les\psi_1(x)$, $x\in\dbR^n$.

For the remaining inequality, we choose a smooth approximation $\psi_2^\delta$ such that
$\psi_2^\delta\les \psi_2\les\psi_1$ and
$\psi_2^\delta\to\psi_2$ uniformly
 as  $\delta\to0.
$
Since
$ \ds
P_\varepsilon^{W_\varepsilon}(x)
\ges \frac1\varepsilon\psi_2^\delta(x),
$ $ x\in\mathbb R^n,
$
repeating the above argument for
$Y_s^{\varepsilon,x;u,v}-\psi_2^\delta(X_s^{x;u,v})$ yields
$\ds
W_\varepsilon(x)-\psi_2^\delta(x)
\ges
-\frac{\varepsilon C_\delta}{\lambda\varepsilon+1},$ $ x\in\dbR^n.
$
Letting $\varepsilon\to0$ and then $\delta\to0$, we obtain
$ \ds
\psi_2(x)\les
\liminf_{\varepsilon\to0}W_\varepsilon(x),
$ $x\in\mathbb R^n.
$

 \end{proof}

\begin{proposition}\label{convergence-VI-HJB}\sl
Under assumptions {\bf(C1)}-{\bf(C5)}, the unique viscosity solution $W_{\e}(\cd)$ of  \eqref{HJB equation-penalized value function} converges locally uniformly to a continuous function $W(\cd)\in C_b(\mathbb{R}^n)$ as $\e\rightarrow0$.
Moreover, the limit  $W(\cd)$ is the unique viscosity solution of the  variational inequality \eqref{VI-}, and satisfies $\psi_2\les W\les\psi_1$.
\end{proposition}

\begin{proof}
We begin by defining the following  lower and upper half-limits, for $x\in\dbR^n$,
 \begin{equation}\label{v+-definition}
W_\ast (x):=\mathop{\liminf}\limits_{(y,\e)\rightarrow(x,0)}W_{\e}(y),\q W^\ast (x):=\mathop{\limsup}\limits_{(y,\e)\rightarrow(x,0)}W_{\e}(y).
\end{equation}
 By the uniform estimates  obtained in the proof of Lemma  \ref{penalized value function-uniformly boundness}, these
half-relaxed limits are finite-valued and satisfy
\bel{psi-W+-}\psi_2\les W_\ast \les W^\ast \les\psi_1.\ee

We now show that    $W^\ast $ is a viscosity subsolution of \eqref{VI-}.   To this end, take any test function
$\varphi\in C_{l,b}^{3}(\mathbb{R}^n)$  such that $W^\ast -\varphi$ attains a strict local maximum at a point $y$. Without loss of generality, we assume $W^\ast (y)=\varphi(y)$.

Combined with \eqref{v+-definition}, we know there exists some sequence $\{(\e_n, y_n)\}_{n\ges 1}$ with $\e_n > 0$, $\e_n \to 0$, and $y_n \to y$ as $n\to\i$, such that, after passing to a subsequence if necessary, the following hold

(i)  $W_{\e_n}-\varphi\les W_{\e_n}(y_n)-\varphi(y_n)$   in a neighbourhood of $y_n$, for all $n\ges 1$;

(ii)  $W_{\e_n}(y_n)\to W^\ast (y)$, as  $n\to\i$.
\\
Since $W_{\e}$ is a viscosity subsolution of \eqref{HJB equation-penalized value function},  the following inequality holds,
\bel{W-e-sub}\ba{ll}
\ns\ds
 \lambda W_{{\e_n}}(y_n)+\frac{1}{{\e_n}}\big(W_{{\e_n}}(y_n)-\psi_1(y_n)\big)^{+}
-\frac{1}{\e_n}\big(W_{\e_n}(y_n)-\psi_2(y_n)\big)^{-}\\
\ns\ds
-H^{-}\big(y_n, \lambda W_{{\e_n}}(y_n), D\varphi(y_n),D^2\varphi(y_n)\big)\les0.
\ea\ee
%
Then we claim that,   letting $n\to\i$ yields
\bel{W+-sub}
\min\Big\{ \max\{ \lambda W^\ast (y)
-H^{-}\big(y, \lambda W^\ast(y), D\varphi(y),D^2\varphi(y)\big), W^\ast (y)-\psi_1(y)\}, W^\ast (y)-\psi_2(y)\Big\}\les0,
\ee
which  implies $W^\ast $ to be  a viscosity subsolution of \eqref{VI-}.

To justify the limit, we consider two cases based on \eqref{psi-W+-}, we discuss  two cases: (a) $W^\ast (y)-\psi_2(y)=0$; and  (b) $W^\ast (y)-\psi_2(y)>0$.

 In case (a), \eqref{W+-sub} holds trivially.
For case (b),   we have some positive integer   $N$ such that $ W_{\e_n}(y_n)-\psi_2(y_n)>0$ whenever $n>N$. Consequently,  when $n>N,$ \eqref{W-e-sub} simplifies to
\bel{W-e-sub-2}
 \lambda W_{{\e_n}}(y_n)-H^{-}(y_n,\lambda W_{{\e_n}}(y_n),D\varphi(y_n),D^2\varphi(y_n))\les  -\frac{1}{{\e_n}}(W_{{\e_n}}(y_n)-\psi_1(y_n))^{+}
\les 0.
\ee
Letting $n\to\infty$ yields
$$
\lambda W^*(y)-H^-(y, \lambda W^\ast(y), D\varphi(y),D^2\varphi(y))\les0.
$$
Together with $W^*(y)\les\psi_1(y)$ from  \eqref{psi-W+-}, this gives \eqref{W+-sub}.

A similar analysis shows that  $W_\ast $ is a viscosity supersolution of \eqref{VI-}.
By the comparison principle in  Theorem
 \ref{HJBI-VI-solution}, we know $W^\ast \les W_\ast $.
 Thus, $W:=W^\ast =W_\ast  $ is continuous and  the    viscosity solution of \eqref{VI-}.   Moreover, it satisfies $\psi_2\les W\les\psi_1$.
The local uniform convergence
$W_\varepsilon\to W$ on $\mathbb R^n$ follows from the standard
half-relaxed limits argument. Finally, the uniqueness follows from
 Theorem   \ref{HJBI-VI-solution}.
\end{proof}

\subsection{Dynamic programming principle for the  penalized SDG}

To identify the viscosity solution obtained in Proposition \ref{convergence-VI-HJB}  with the lower
value function $W^-$ in \eqref{LU-VF}, we need a DPP for the penalized
SDG.   Since such a DPP involves the value function evaluated at intermediate
states at stopping times,   it is necessary to extend the stochastic
representation of $W_\varepsilon$ from deterministic initial states to random
initial states.
To this end,
for $(t,x)\in [0,\i)\times\dbR^n$, $u(\cd)\in\cU_{t,\i}$,  $v(\cd)\in\cV_{t,\i}$,  with  $W_\e$ defined  in \eqref{penalized BSDE}, we consider the following infinite horizon BSDE,
     \begin{equation}\label{penalized BSDE-auxiliary 3}
  \ba{ll}
 \ns\ds \cY_{s}^{\e,t,x;u,v}
=\cY_{T}^{\e,t,x;u,v} +\int_s^T \big( f(X_{r}^{t,x;u,v},\l\cY_r  ^{\e,t,x;u,v}, \cZ_r  ^{\e,t,x;u,v} ,u_r,v_r)   + Q^{W_\e}_{\e}(X_{r}^{t,x;u,v})    -   \l  \cY_r ^{\e,t,x;u,v}\big)\,\md r\\
 \ns\ds \qquad\qquad\q
-\int_s^T\cZ_r^{\e,t,x;u,v}  \md  B_r, \q  t\les s\les T <\i,
    \ea
 \end{equation}
 where $X_s^{t,x;u,v}$ denotes the solution of \eqref{state} starting
from $x$ at time $t$, and
\bel{P-1}\ds  Q^{W_\e}_{\e}(x ):=
-\frac{1}{\e}\big(W_\e(x)-\psi_1(x)\big)^{+}
+\frac{1}{\e}\big(W_\e(x)-\psi_2(x)\big)^{-},
\q x\in\dbR^n.
\ee
Under assumptions {\bf(C1)}-{\bf(C5)}, the BSDE \eqref{penalized BSDE-auxiliary 3} is
well-posed. Then, we define the corresponding cost functional by
  $J_\varepsilon(t,x;u,v):=\cY_{t}^{\e,t,x;u,v}$.
The lower value function associated with \eqref{penalized BSDE-auxiliary 3} is then given by
\bel{cW-kappa}
\mathcal W_\varepsilon(t,x)
:=
\essinf_{\beta\in\mathcal B_{t,\infty}}
\esssup_{u\in\mathcal U_{t,\infty}}
 \cY_t^{\varepsilon,t,x;u,\beta(u)},
\qquad
(t,x)\in[0,\infty)\times\mathbb R^n .
\ee

Applying  \cite[Theorem 3.5]{Huang-Wei-2026}  to the above infinite horizon SDG,
we know that $\mathcal W_\varepsilon$ is deterministic and is a bounded
viscosity solution of the HJBI equation associated with
\eqref{penalized BSDE-auxiliary 3}. Since the coefficients are
time-homogeneous and the discount coefficient is constant,
\cite[Corollary 3.10]{Huang-Wei-2026} yields that this non-autonomous HJBI
equation admits a stationary reduction. Moreover, the resulting stationary
equation is precisely \eqref{HJB equation-penalized value function}. Hence, by
the uniqueness of the viscosity solution,
\bel{W-e-cY-t}
\mathcal W_\varepsilon(t,x)=W_\varepsilon(x),
\qquad
(t,x)\in[0,\infty)\times\mathbb R^n .
\ee
Although the dependence on $t$ is only notational, we keep the explicit time
variable for later use in the dynamic programming argument.

\begin{lemma}\label{Le-W-e-cY-ran}\sl  Under assumptions  {\bf(C1)}-{\bf(C5)},  for any  $\rho\in\cS$  with $\rho<\i$, $\dbP$-a.s., and  any $\mathcal F_\rho$-measurable $\mathbb R^n$-valued random variable $\xi$ with $|\xi|<\infty$, $\mathbb P$-a.s., we have
     \bel{W-e-cY-random}W_\e(\xi)=\cW_\e(\rho,\xi)=
\essinf_{\beta\in\mathcal B_{\rho,\infty}}
\esssup_{u\in\mathcal U_{\rho,\infty}}
\mathcal Y_{\rho}^{\varepsilon,\rho,\xi;u,\beta(u)},
\qquad \mathbb P\text{-a.s.}\ee
       \end{lemma}

\begin{proof}
\emph{Step 1}.
We first prove the second equality. Let $\rho\in\cS$ be bounded and let $\xi\in L^2_{\mathcal F_\rho}(\Omega;\mathbb R^n)$. Then by an approximation argument similar to that in
\cite[Proposition 3.16]{Li-Li-Wei-2021}, one obtains the following identity: \begin{equation}\label{eq:auxiliary-DPP-backward-semigroup} \cW_{\e}(\rho,\xi) = \mathop{\mathrm{essinf}}\limits_{\b\in \cB_{\rho,\i }} \mathop{\mathrm{esssup}}\limits_{u\in \cU_{\rho,\i}} \cY_\rho^{\varepsilon,\rho,\xi;u,\beta(u)}, \q \mathbb P\mbox{-a.s.} \end{equation}
   It remains to extend \eqref{eq:auxiliary-DPP-backward-semigroup} to the case where $\rho$ is unbounded but finite $\mathbb P$-a.s., and $\xi$ is $\mathcal F_\rho$-measurable with $|\xi|<\infty$, $\mathbb P$-a.s.

Fix $x_0\in\mathbb R^n$. For $m\ges1$, define
$
\rho_m:=\rho\wedge m,$ $
A_m:=\{\rho\les m,\ |\xi|\les m\},
$
and
$
\xi_m:=\xi{\bf1}_{A_m}+x_0{\bf1}_{A_m^c}.
$
Since $\xi$ is $\mathcal F_\rho$-measurable, we have $ A_m\in\mathcal F_\rho. $ Moreover, $A_m\in\mathcal F_{\rho_m}$.
 Indeed, for $t<m$, $ A_m\cap\{\rho_m\les t\} = \{|\xi|\les m\}\cap\{\rho\les t\}\in\mathcal F_t, $ while for $t\ges m$, $ A_m\cap\{\rho_m\les t\} = A_m = \{|\xi|\les m\}\cap\{\rho\les m\}\in\mathcal F_m\subset\mathcal F_t. $ The same criterion also shows that $\xi{\bf1}_{A_m}$ is $\mathcal F_{\rho_m}$-measurable. Hence $ \xi_m\in L^\infty_{\mathcal F_{\rho_m}}(\Omega;\mathbb R^n) \subset L^2_{\mathcal F_{\rho_m}}(\Omega;\mathbb R^n). $ Furthermore, $ (\rho_m,\xi_m)=(\rho,\xi)$ on $A_m. $ Since $\rho<\infty$ and $|\xi|<\infty$, $\dbP$-a.s., we have $ A_m\uparrow\Omega $ as $m\to\infty$, up to a $\dbP$-null set.
Therefore, applying \eqref{eq:auxiliary-DPP-backward-semigroup} to $(\rho_m,\xi_m)$, we have
\begin{equation}\label{eq:DPP-random-truncated}
\mathcal W_\varepsilon(\rho_m,\xi_m)
=
\essinf_{{  \beta}\in\mathcal B_{\rho_m,\infty}}
\esssup_{u\in\mathcal U_{\rho_m,\infty}}
\mathcal Y_{\rho_m}^{\varepsilon,\rho_m,\xi_m;u,\beta(u)},
\qquad \dbP\text{-a.s.}
\end{equation}

Next, we claim that
\begin{equation}\label{eq:DPP-random-locality}{\bf1}_{A_m}
\essinf_{\hat{\beta}\in\mathcal B_{\rho_m,\infty}}
\esssup_{\hat{u}\in\mathcal U_{\rho_m,\infty}}
\mathcal Y_{\rho_m}^{\varepsilon,\rho_m,\xi_m;\hat{u},\hat{\beta}(\hat{u})}
  =
{\bf1}_{A_m}
\essinf_{\beta\in\mathcal B_{\rho,\infty}}
\esssup_{u\in\mathcal U_{\rho,\infty}}
\mathcal Y_{\rho}^{\varepsilon,\rho,\xi;u,\beta(u)} .
\end{equation}
Fix $u_0\in U$ and $v_0\in V$. Let
$\hat{\beta}\in\mathcal B_{\rho_m,\infty}$ be fixed. For
$u\in\mathcal U_{\rho,\infty}$, define an extension $\tilde u$ of $u$ on
$[0,\infty)$ by $
\tilde u_s:=u_0{\bf1}_{\{s<\rho\}}+u_s{\bf1}_{\{s\ges\rho\}},
$ $ s\ges0.
$
Then $\tilde u\in\mathcal U_{0,\infty}$.
Define
$\bar u\in\mathcal U_{\rho_m,\infty}$ by
$
\bar u_s:=\tilde u_s{\bf1}_{A_m}+u_0{\bf1}_{A_m^c},
$ $ s\ges \rho_m.
$
Next, define a map $\beta:\mathcal U_{\rho,\infty}\to
\mathcal V_{\rho,\infty}$ by
$$
\beta(u)_s:
=
\hat{\beta}(\bar u)_s{\bf1}_{A_m}+v_0{\bf1}_{A_m^c},
\qquad s\ges \rho .
$$
Since $A_m\in\mathcal F_{\rho_m}\cap\mathcal F_\rho$, the processes
$\bar u$ and $\beta(u)$ are admissible on $\llbracket\rho_m,\infty\rrbracket$ and
$\llbracket\rho,\infty\rrbracket$, respectively.
We now verify that $\beta$ is nonanticipative.
Let
$S\in\mathcal S$ with $S\ges\rho$, $\dbP$-a.s.,
and let $u^1,u^2\in\mathcal U_{\rho,\infty}$ satisfy
$
u^1\equiv u^2
$ on $\llbracket\rho,S\rrbracket .
$
Then the corresponding
controls $\bar u^1,\bar u^2$ coincide on $\llbracket\rho_m,S\rrbracket$: on $A_m$ we have
$\rho_m=\rho$ and
$\bar u^i=u^i$, $i=1,2$, while on $A_m^c$ both controls are equal to $u_0$. Since
$\hat{\beta} $ is nonanticipative, $
\hat{\beta}(\bar u^1)
\equiv
\hat{\beta}(\bar u^2)
$ on $\llbracket\rho_m,S\rrbracket .
$
Consequently,
$
\beta(u^1)\equiv\beta(u^2)
$ on $\llbracket\rho,S\rrbracket .
$
This proves that $\beta\in\mathcal B_{\rho,\infty}$.

On $A_m$, we have $\rho_m=\rho$, $\xi_m=\xi$, and, on
$\llbracket\rho,\infty\rrbracket$, $\bar u=u$ and $\hat{\beta}(\bar u)=\beta(u)$.
 Hence, by the   uniqueness of the
controlled SDE and BSDE,
$$
{\bf1}_{A_m}
\mathcal Y_{\rho}^{\varepsilon,\rho,\xi;u,\beta(u)}   ={\bf1}_{A_m}
\mathcal Y_{\rho_m}^{\varepsilon,\rho_m,\xi_m;\bar u,\hat{\beta}(\bar u)}\les
{\bf1}_{A_m}
\esssup_{\hat{u}\in\mathcal U_{\rho_m,\infty}}
\mathcal Y_{\rho_m}^{\varepsilon,\rho_m,\xi_m;\hat{u},\hat{\beta}(\hat{u})}
.
$$
Taking the essential supremum over $ {u}\in\mathcal U_{\rho,\infty}$ yields
$$
{\bf1}_{A_m}
\esssup_{u\in\mathcal U_{\rho,\infty}}
\mathcal Y_{\rho}^{\varepsilon,\rho,\xi;u,\beta(u)}
\les
{\bf1}_{A_m}
\esssup_{\hat{u}\in\mathcal U_{\rho_m,\infty}}
\mathcal Y_{\rho_m}^{\varepsilon,\rho_m,\xi_m;\hat{u},\hat{\beta}(\hat{u})}.
$$
Since the above inequality holds for the strategy $\beta$ constructed from the
fixed $\hat{\beta}$, we obtain
$$
{\bf1}_{A_m}
\essinf_{\beta\in\mathcal B_{\rho,\infty}}
\esssup_{u\in\mathcal U_{\rho,\infty}}
\mathcal Y_{\rho}^{\varepsilon,\rho,\xi;u,\beta(u)}
\les
{\bf1}_{A_m}
\esssup_{\hat{u}\in\mathcal U_{\rho_m,\infty}}
\mathcal Y_{\rho_m}^{\varepsilon,\rho_m,\xi_m;\hat{u},\hat{\beta}(\hat{u})}.
$$
Then by the arbitrariness of  $\hat{\beta}\in\mathcal B_{\rho_m,\infty}$, we get
\begin{equation}\label{eq:DPP-random-locality-ineq1}
\begin{aligned}
&{\bf1}_{A_m}
\essinf_{\beta\in\mathcal B_{\rho,\infty}}
\esssup_{u\in\mathcal U_{\rho,\infty}}
\mathcal Y_{\rho}^{\varepsilon,\rho,\xi;u,\beta(u)}
 \les
{\bf1}_{A_m}
\essinf_{\hat{\beta}\in\mathcal B_{\rho_m,\infty}}
\esssup_{\hat{u}\in\mathcal U_{\rho_m,\infty}}
\mathcal Y_{\rho_m}^{\varepsilon,\rho_m,\xi_m;\hat{u},\hat{\beta}(\hat{u})}.
\end{aligned}
\end{equation}

Conversely, fix $\beta\in\mathcal B_{\rho,\infty}$. For
$\hat u\in\mathcal U_{\rho_m,\infty}$, define
$u^\sharp\in\mathcal U_{\rho,\infty}$ by
$
u^\sharp_s:
=
\hat u_s{\bf1}_{A_m}+u_0{\bf1}_{A_m^c},
$ $s\ges \rho .
$
Let $\widetilde{\beta(u^\sharp)}$ be the extension of $\beta(u^\sharp)$ to
$[0,\infty)$ by $
\widetilde{\beta(u^\sharp)}_s
:=
v_0{\bf1}_{\{s<\rho\}}
+
\beta(u^\sharp)_s{\bf1}_{\{s\ges\rho\}},
$ $ s\ges0.
$
 Define
$\hat\beta:\mathcal U_{\rho_m,\infty}\to\mathcal V_{\rho_m,\infty}$ by
$$
\hat\beta(\hat u)_s:
=
\widetilde{\beta(u^\sharp)}_s{\bf1}_{A_m}
+
v_0{\bf1}_{A_m^c},
\qquad s\ges \rho_m .
$$
By the similar  argument as above, using
$A_m\in\mathcal F_{\rho_m}\cap\mathcal F_\rho$, we have
$\hat\beta\in\mathcal B_{\rho_m,\infty}$.

On $A_m$, we have $\rho_m=\rho$, $\xi_m=\xi$, and, on
$\llbracket\rho,\infty\rrbracket$, $u^\sharp=\hat u$ and
$\hat\beta(\hat u)=\beta(u^\sharp)$. Hence, by the  uniqueness of
the controlled SDE and BSDE,
$$
{\bf1}_{A_m}
\mathcal Y_{\rho_m}^{\varepsilon,\rho_m,\xi_m;\hat u,\hat\beta(\hat u)}
=
{\bf1}_{A_m}
\mathcal Y_{\rho}^{\varepsilon,\rho,\xi;u^\sharp,\beta(u^\sharp)} \les
{\bf1}_{A_m}
\esssup_{u\in\mathcal U_{\rho,\infty}}
\mathcal Y_{\rho}^{\varepsilon,\rho,\xi;u,\beta(u)} .
$$
Then taking the essential supremum
over $\hat u\in\mathcal U_{\rho_m,\infty}$ gives
$$
{\bf1}_{A_m}
\esssup_{\hat u\in\mathcal U_{\rho_m,\infty}}
\mathcal Y_{\rho_m}^{\varepsilon,\rho_m,\xi_m;\hat u,\hat\beta(\hat u)}
\les
{\bf1}_{A_m}
\esssup_{u\in\mathcal U_{\rho,\infty}}
\mathcal Y_{\rho}^{\varepsilon,\rho,\xi;u,\beta(u)} .
$$

Using $\hat\beta\in\mathcal B_{\rho_m,\infty}$, it follows that
$$
{\bf1}_{A_m}
\essinf_{\hat\beta\in\mathcal B_{\rho_m,\infty}}
\esssup_{\hat u\in\mathcal U_{\rho_m,\infty}}
\mathcal Y_{\rho_m}^{\varepsilon,\rho_m,\xi_m;\hat u,\hat\beta(\hat u)}
\les
{\bf1}_{A_m}
\esssup_{u\in\mathcal U_{\rho,\infty}}
\mathcal Y_{\rho}^{\varepsilon,\rho,\xi;u,\beta(u)} .
$$
Taking the essential infimum over
$\beta\in\mathcal B_{\rho,\infty}$ yields
\begin{equation}\label{eq:DPP-random-locality-ineq2}
{\bf1}_{A_m}
\essinf_{\hat\beta\in\mathcal B_{\rho_m,\infty}}
\esssup_{\hat u\in\mathcal U_{\rho_m,\infty}}
\mathcal Y_{\rho_m}^{\varepsilon,\rho_m,\xi_m;\hat u,\hat\beta(\hat u)}
 \les
{\bf1}_{A_m}
\essinf_{\beta\in\mathcal B_{\rho,\infty}}
\esssup_{u\in\mathcal U_{\rho,\infty}}
\mathcal Y_{\rho}^{\varepsilon,\rho,\xi;u,\beta(u)} .
\end{equation}
Combining \eqref{eq:DPP-random-locality-ineq1} and \eqref{eq:DPP-random-locality-ineq2}, \eqref{eq:DPP-random-locality} is proved.

Combining \eqref{eq:DPP-random-truncated} with
\eqref{eq:DPP-random-locality}, we obtain
$$
{\bf1}_{A_m}\mathcal W_\varepsilon(\rho_m,\xi_m)
 =
{\bf1}_{A_m}
\essinf_{\beta\in\mathcal B_{\rho,\infty}}
\esssup_{u\in\mathcal U_{\rho,\infty}}
\mathcal Y_{\rho}^{\varepsilon,\rho,\xi;u,\beta(u)}  ,
\qquad \dbP\text{-a.s.}
$$
On the other hand, since $(\rho_m,\xi_m)=(\rho,\xi)$ on $A_m$, we have
$
{\bf1}_{A_m}\mathcal W_\varepsilon(\rho_m,\xi_m)
=
{\bf1}_{A_m}\mathcal W_\varepsilon(\rho,\xi),
$ $ \dbP\text{-a.s.}
$
Therefore,
\begin{equation}\label{eq:DPP-random-on-Am}
\begin{aligned}
{\bf1}_{A_m}\mathcal W_\varepsilon(\rho,\xi)
=
{\bf1}_{A_m}
\essinf_{\beta\in\mathcal B_{\rho,\infty}}
\esssup_{u\in\mathcal U_{\rho,\infty}}
\mathcal Y_{\rho}^{\varepsilon,\rho,\xi;u,\beta(u)},
\qquad \dbP\text{-a.s.}
\end{aligned}
\end{equation}
Since \eqref{eq:DPP-random-on-Am} holds for every $m$ and $A_m\uparrow\Omega$ up to a $\mathbb P$-null set, the equality holds on $\ds\bigcup_{m\ges1}A_m$, outside a $\mathbb P$-null set. Therefore,
\begin{equation}\label{eq:DPP-random-unbounded}
\mathcal W_\varepsilon(\rho,\xi)
=
\essinf_{\beta\in\mathcal B_{\rho,\infty}}
\esssup_{u\in\mathcal U_{\rho,\infty}}
\mathcal Y_{\rho}^{\varepsilon,\rho,\xi;u,\beta(u)},
\qquad \dbP\text{-a.s.}
\end{equation}
This proves the desired second equality for  random initial
pairs $(\rho,\xi)$ with $\rho<\infty$, $|\xi|<\infty$, $\dbP$-a.s.

  \ss

\emph{Step 2}.
It remains to identify $\mathcal W_\varepsilon(\rho,\xi)$ with
$W_\varepsilon(\xi)$. Since both $W_\varepsilon$ and
$\mathcal W_\varepsilon$ are deterministic continuous functions, and since
the pointwise identity \eqref{W-e-cY-t} holds for every deterministic initial
pair $(t,x)\in[0,\infty)\times\mathbb R^n$, we have, for every
$\omega\in\{\rho<\infty,\ |\xi|<\infty\}$,
$
\mathcal W_\varepsilon(\rho(\omega),\xi(\omega))
=
W_\varepsilon(\xi(\omega)).
$
Since $\rho<\infty$ and $|\xi|<\infty$, $\dbP$-a.s., it follows that
$
\mathcal W_\varepsilon(\rho,\xi)
=
W_\varepsilon(\xi),
$ $\dbP$-a.s.
Combining this with \eqref{eq:DPP-random-unbounded}, we obtain
\eqref{W-e-cY-random}.
\end{proof}

We note that Proposition \ref{lower value function existence} identifies
$W_\varepsilon$ as the unique viscosity solution of the HJBI equation
\eqref{HJB equation-penalized value function}. This identification, however,
does not rely on the standard dynamic programming argument. Nevertheless, the
DPP for $W_\varepsilon$ remains valid and will be established below. To state
it precisely, we first introduce the discounted backward semigroup associated
with the approximating game.

\bde\label{Def-SemiG-1}\sl
For any initial state $x\in\mathbb R^n$, any $\dbF$-stopping times $\rho$ and $\varsigma$ satisfying $0\les \rho\les \varsigma$, $\mathbb P$-a.s., any terminal variable $ \eta\in L^2_{\mathcal F_\varsigma}(\Omega;\mathbb R)$, and any admissible controls $u\in\mathcal U_{0,\varsigma}$ and $v\in\mathcal V_{0,\varsigma}$, we define
$$
G_{\rho,\varsigma}^{\varsigma,\varepsilon,x;u,v}\big[e^{-\lambda\varsigma}\eta\big]
:=
e^{-\lambda \rho }\tilde Y_\rho^{\varsigma,\varepsilon,x;u,v},
$$
where the pair $\bigl(\tilde Y^{\varsigma,\varepsilon,x;u,v},\tilde Z^{\varsigma,\varepsilon,x;u,v}\bigr)$ is the solution of the following BSDE on $\llbracket 0,\varsigma\rrbracket $,
 \begin{equation}\label{semigroups-BSDE-1}\left\{\2n
\begin{split}
\tilde{Y}_{s\wedge\varsigma}^{\varsigma,\e,x;u,v}
&=
\tilde{Y}_{T\wedge\varsigma}^{\varsigma,\e,x;u,v}
+
\int_{s\wedge\varsigma}^{T\wedge\varsigma}
\big(
f(X_{r}^{x;u,v}, \lambda \tilde{Y}_{r}^{\varsigma,\e,x;u,v},
\tilde{Z}_{r}^{\varsigma,\e,x;u,v},
u_r,v_r)
+
Q_{\e}^{W_\e}(X_{r}^{x;u,v})
\\
&\quad
-
\lambda \tilde{Y}_{r}^{\varsigma,\e,x;u,v}
\big)\,\md r
-
\int_{s\wedge\varsigma}^{T\wedge\varsigma}
\tilde{Z}_{r}^{\varsigma,\e,x;u,v}\,\md B_r,
\qquad \forall\  0\les s\les T<\infty,                                    \\
\tilde{Y}_{\varsigma}^{\varsigma,\e,x;u,v}
&=\eta,\qquad \text{on } \{\varsigma<\infty\},
\end{split}\right.
\end{equation}
and $X^{x;u,v}$ is the solution of \eqref{state}.
\ede

 \br \label{Re-discounted-semigroup}\sl

The notation in Definition \ref{Def-SemiG-1} is justified by the discounted form of the
BSDE. Indeed, applying It\^o's formula to
$e^{-\lambda s}\tilde Y_s^{\varsigma,\e,x;u,v}$ on
$\llbracket \rho,\varsigma\rrbracket$, we obtain
$$
\begin{aligned}
e^{-\lambda\rho}\tilde Y_\rho^{\varsigma,\e,x;u,v}
&=
e^{-\lambda\varsigma}\eta\,\mathbf 1_{\{\varsigma<\infty\}}
+
\int_\rho^\varsigma
e^{-\lambda r}
\big(
f\big(X_r^{x;u,v}, \lambda \tilde{Y}_{r}^{\varsigma,\e,x;u,v}, \tilde Z_r^{\varsigma,\e,x;u,v},u_r,v_r\big)
+
Q_\varepsilon^{W_\varepsilon}
\big(X_r^{x;u,v}\big)
\big)\,\md r  \\
&\quad
-
\int_\rho^\varsigma
e^{-\lambda r}
\tilde Z_r^{\varsigma,\e,x;u,v}\,\md B_r .
\end{aligned}
$$
Note that, by convention,
$
e^{-\lambda\varsigma}\eta=0
$ $\text{on } \{\varsigma=\infty\}.
$
Thus the discount factor eliminates the linear term
$-\lambda\tilde Y_r^{\varsigma,\e,x;u,v}$, and the terminal contribution is
given by the discounted value $e^{-\lambda\varsigma}\eta$. This explains the
notation
$\ds
 G_{\rho,\varsigma}^{\varsigma,\varepsilon,x;u,v}
\big[e^{-\lambda\varsigma}\eta\big]
=
e^{-\lambda\rho}\tilde Y_\rho^{\varsigma,\e,x;u,v}.
$
\er

Then, the strong DPP for  $W_\e$ is established as follows.

\begin{proposition}\label{DPP}\sl
Assume {\bf(C1)}-{\bf(C5)}. Then,  for any  $\varrho\in\cS$ and $x\in \mathbb{R}^n$,
 \begin{equation}\label{auxiliary DPP}
W_{\e}(x)=\mathop{\mathrm{essinf}}\limits_{\b\in \cB_{0,\varrho }}\mathop{\mathrm{esssup}}\limits_{u\in \cU_{0,\varrho}} G_{0,\varrho}^{ \varrho,\e,x;u,\b(u)}
\big[e^{-\l \varrho} W_{\e}(X_{\varrho}^{x;u,\b(u)})\big],\ \mathbb{P}\mbox{-a.s.},
      \end{equation}
where, by convention,
$
e^{-\lambda\varrho}W_\varepsilon\bigl(X_\varrho^{x;u,\beta(u)}\bigr)=0
$ on $\{\varrho=\infty\}.
$
\end{proposition}

\begin{proof}

By \eqref{W-e-cY-t}, we have
$
W_\e(x)=\cW_\e(0,x),$ $ x\in\dbR^n.
$
Hence it is enough to prove \eqref{auxiliary DPP} for $\cW_\e(0,x)$.
For simplicity, in the sequel we write
$
\cW_\e(x):=\cW_\e(0,x),$ $x\in\dbR^n .
$
We shall also use the fact that $|W_\e(x)|\les B_\e$ for all
$x\in\dbR^n$, and the same bound holds for $\cW_\e(x)$.

\ms

\no {\emph{Step 1. The case of  deterministic times.}}
By  \cite[Proposition 4.1]{Buckdahn-Li-Zhao-2021},  for any  $t\ges0$, $x\in \mathbb{R}^n$, we have
 \begin{equation*}\label{auxiliary DPP-backward semigroup-t-BW}
\cW_{\e}(x)=\mathop{\mathrm{essinf}}\limits_{\b\in \cB_{0,t}}\mathop{\mathrm{esssup}}\limits_{u\in \cU_{0,t}}\cG_{0,t}^{\e,x;u,\b(u)}\big[\cW_{\e}(X_{t}^{x;u,\b(u)})\big],\q \mathbb{P}\mbox{-a.s.},
      \end{equation*}
      where  $\cG_{0,t}^{\e,x;u,\b(u)}\big[\cW_{\e}(X_{t}^{x;u,\b(u)})\big]
      ={\BY} ^{\e,x;u,\b(u)}_0  $, and ${\BY} ^{\e,x;u,v}$ is the solution of the following BSDE:
\begin{equation*}\label{auxiliary penalized backwards semigroups-BSDE}\left\{\2n
  \begin{split}
 & \md    \BY_{s}^{\e,x;u,v}
=-\big( f(X_{s}^{x;u,v},\l{\BY}_{s}^{\e,x;u,v}, {\BZ}_{s}^{\e,x;u,v},u_s,v_s)   + Q^{W_{\e}}_{\e}(X_{s}^{x;u,v})  -  \lambda
 \BY _{s}^{\e,x;u,v}\big)\md s\\
&\qq\qq\q
  + {\BZ}_{s}^{\e,x;u,v}\,\md B_s, \ s\in[0,t],\\
&  {\BY}_{t}^{\e,x;u,v}=\cW_{\e}(X_{t}^{x;u,v}).
    \end{split}\right.
      \end{equation*}
Here  $\cG_{0,t}^{\e,x;u,v}[\cd]$ denotes the usual backward semigroup. Moreover, by the  direct computation, we know
  $  {\BY} ^{\e,x;u,v}_0 = G_{0,t}^{t,\e,x;u,v} \big[e^{-\lambda t }\cW_{\e}(X_{t}^{x;u,v})\big],$ $ \mathbb{P}\mbox{-a.s.} $
Therefore, for all $x\in\dbR^n$,
   \begin{equation}\label{DPP-DT}
\cW_{\e}(x)=\mathop{\mathrm{essinf}}\limits_{\b\in \cB_{0,t}}\mathop{\mathrm{esssup}}\limits_{u\in \cU_{0,t}} G_{0,t}^{ t,\e,x;u,\b(u)}\big[e^{-\lambda t }\cW_{\e}(X_{t}^{x;u,\b(u)})\big],\q \mathbb{P}\mbox{-a.s.}
      \end{equation}

\no {\emph{Step 2. The case of bounded stopping times.}}
By combining   \eqref{DPP-DT} with the approximation argument for bounded stopping times in \cite[Theorem 3.17]{Li-Li-Wei-2021}, we obtain that, for any bounded $\dbF$-stopping time $\varrho$ and  $x\in\mathbb R^n$,
 \begin{equation*}\label{DPP-BS}
\cW_{\e}(x)=\mathop{\mathrm{essinf}}\limits_{\b\in \cB_{0,\varrho}}\mathop{\mathrm{esssup}}\limits_{u\in \cU_{0,\varrho}}G_{0,\varrho}^{ \varrho,\e,x;u,\b(u)}\big[e^{-\lambda \varrho }\cW_{\e}(X_{\varrho}^{x;u,\b(u)})\big],\q \mathbb{P}\mbox{-a.s.}
      \end{equation*}

\no {\emph{Step 3. The case of  arbitrary stopping times.}}
 Let $\varrho$ be an arbitrary $\dbF$-stopping time and set
$
\varrho_n:=\varrho\wedge n,$ $ n\ges1.
$
Then, for  each $n\ges 1$, $\varrho_n$ is   bounded. Therefore, for every $n\ges 1$, $x\in\dbR^n$,
$$
\cW_\varepsilon(x)
=
\essinf_{\beta\in\mathcal B_{0,\varrho_n}}
\esssup_{u\in\mathcal U_{0,\varrho_n}}
G^{  \varrho_n ,\varepsilon,x;u,\beta(u)}_{0,\varrho_n}
\big[e^{-\lambda \varrho_n }\cW_\varepsilon\bigl(X_{\varrho_n}^{x;u,\beta(u)}\bigr)\big],\q \dbP\mbox{-a.s.}
$$

We claim that,  for $n\ges 1$,   $x\in\dbR^n$,
\bel{W-e-re-ex}
\cW_\varepsilon(x)
=
\essinf_{\beta\in\mathcal B_{0,\varrho}}
\esssup_{u\in\mathcal U_{0,\varrho}}
G^{ \varrho_n,\varepsilon,x;u,\beta(u)}_{0,\varrho_n}
\big[e^{-\lambda \varrho _n }\cW_\varepsilon\bigl(X_{\varrho_n}^{x;u,\beta(u)}\bigr)\big],\q \dbP\mbox{-a.s.}
\ee
Indeed, since
$\varrho_n\les \varrho$, every control $u\in\mathcal U_{0,\varrho}$ can be restricted
to $\llbracket 0,\varrho_n\rrbracket$, and every control in
$\mathcal U_{0,\varrho_n}$ can be extended to $\mathcal U_{0,\varrho}$ by fixing an
arbitrary value after $\varrho_n$. The same restriction-extension argument applies to strategies. More precisely,
any $\beta\in\mathcal B_{0,\varrho}$ induces a strategy
$\beta^n\in\mathcal B_{0,\varrho_n}$ by restriction, while any
$\beta^n\in\mathcal B_{0,\varrho_n}$ can be extended to some
$\bar\beta\in\mathcal B_{0,\varrho}$ by fixing an arbitrary value
$v_0\in V$ after $\varrho_n$. The nonanticipativity property guarantees that
these restrictions and extensions are well-defined.
Moreover, the backward semigroup
$
G_{0,\varrho_n}^{\varepsilon,x;u,\beta(u)}
\big[
e^{-\lambda\varrho_n}
\mathcal W_\varepsilon\bigl(X_{\varrho_n}^{x;u,\beta(u)}\bigr)
\big]
$
depends only on the restrictions of $u$ and $\beta(u)$ to
$\llbracket 0,\varrho_n\rrbracket$. Hence replacing
$\mathcal U_{0,\varrho_n}$ and $\mathcal B_{0,\varrho_n}$ by
$\mathcal U_{0,\varrho}$ and $\mathcal B_{0,\varrho}$ does not change the value of
the essential infimum-supremum. Therefore,
\eqref{W-e-re-ex} holds.

\ss

Next, we study the limit of \eqref{W-e-re-ex}  as $n\to\infty$.
Consider the following BSDEs on random intervals,
\begin{equation}\label{BY-n-e}\left\{\2n
\begin{aligned}
{\BY}_{s\wedge\varrho_n}^{n,\e;u,v}
&=
{\BY}_{T\wedge\varrho_n}^{n,\e;u,v}
+
\int_{s\wedge\varrho_n}^{T\wedge\varrho_n}
\Big(
f(X_{r}^{x;u,v}, \l{\BY}_{r}^{n,\e;u,v}, {\BZ}_{r}^{n,\e;u,v},u_r,v_r)
+
Q^{W_\e}_{\e}(X_{r}^{x;u,v})
-
\lambda {\BY}_{r}^{n,\e;u,v}
\Big)\md r       \\
&\quad
-
\int_{s\wedge\varrho_n}^{T\wedge\varrho_n}
{\BZ}_{r}^{n,\e;u,v}\md B_r,
\qquad 0\les s\les T<\infty,                                      \\
{\BY}_{\varrho_n}^{n,\e;u,v}
&=
\cW_\varepsilon\bigl(X_{\varrho_n}^{x;u,v}\bigr),
\end{aligned}\right.
\end{equation}
and
\begin{equation}\label{BY-0-e}\left\{\2n
\begin{aligned}
{\BY}_{s\wedge\varrho}^{\e;u,v}
&=
{\BY}_{T\wedge\varrho}^{\e;u,v}
+
\int_{s\wedge\varrho}^{T\wedge\varrho}
\Big(
f(X_{r}^{x;u,v}, \l{\BY}_{r}^{\e;u,v}, {\BZ}_{r}^{\e;u,v},u_r,v_r)
+
Q^{W_\e}_{\e}(X_{r}^{x;u,v})
-
\lambda {\BY}_{r}^{\e;u,v}
\Big)\md r       \\
&\quad
-
\int_{s\wedge\varrho}^{T\wedge\varrho}
{\BZ}_{r}^{\e;u,v}\md B_r,
\qquad 0\les s\les T<\infty,                                      \\
{\BY}_{\varrho}^{\e;u,v}
&=
\cW_\varepsilon\bigl(X_{\varrho}^{x;u,v}\bigr),
\qquad \text{on } \{\varrho<\infty\}.
\end{aligned}\right.
\end{equation}
The well-posedness of these equations follows directly from Lemma
\ref{LemmaA-1-FBSDE}.

 By Definition \ref{Def-SemiG-1} and   the flow property of the   backward semigroup,
 for $(u,v)\in\mathcal U_{0,\varrho}\times\mathcal V_{0,\varrho}$, we have
\bel{G-G}\ba{ll}
\ns\ds
 {\BY}_{0}^{ n,\e ;u,v}- {\BY}_{0}^{\e;u,v }=G^{ \varrho_n, \varepsilon,x;u,v}_{0, \varrho_n}
\big[e^{-\lambda \varrho_n }\cW_\varepsilon\bigl(X_{ \varrho_n}^{x;u,v}\bigr)\big] -  G^{\varrho,\varepsilon,x;u,v}_{0,\varrho}
\big[e^{-\lambda \varrho }\cW_\varepsilon\bigl(X_{\varrho}^{x;u,v}\bigr)\big]\\
 %
 \ns\ds
  = G^{ \varrho_n, \varepsilon,x;u,v}_{0, \varrho_n}
\big[e^{-\lambda \varrho_n }\cW_\varepsilon\bigl(X_{ \varrho_n}^{x;u,v}\bigr)\big] - G^{\varrho_n,\varepsilon,x;u,v}_{0,\varrho_n}\big[ e^{-\lambda \varrho  _n}{\BY}_{\varrho  _n}^{\e;u,v }  \big] ,\q \dbP\mbox{-a.s.}
\ea\ee
%
%
%
For \eqref{BY-n-e} and \eqref{BY-0-e} on $\llbracket 0,\varrho_n\rrbracket $, applying  Lemma  \ref{Pro-A.3}, there exists some probability measure $\dbQ^n$  such that
\bel{Yn-Y}
 | {\BY}_{0}^{ n,\e ;u,v}- {\BY}_{0}^{\e;u,v }|
\les
 \dbE^{\dbQ^n}
 \[
\big|e^{-\lambda \varrho_n}\cW_\varepsilon\bigl(X_{\varrho_n}^{x;u,v}\bigr)
-e^{-\lambda \varrho  _n}{\BY}_{\varrho  _n}^{\e;u,v }
\big|
\] .
\ee

We claim that there exists a constant $C_\varepsilon>0$, independent of
$n,u$ and $v$, such that
\begin{equation}\label{terminal-difference-rhon}
\bigl|e^{-\lambda \varrho_n}\cW_\varepsilon\bigl(X_{\varrho_n}^{x;u,v}\bigr)
-e^{-\lambda \varrho  _n}{\BY}_{\varrho  _n}^{\e;u,v }
\bigr|
\les C_\varepsilon e^{- \lambda n},
\qquad \mathbb P\text{-a.s.}
\end{equation}
Indeed, on $\{ \varrho \les n\}$, we have  $\varrho_n=\varrho $, and therefore
$ e^{-\lambda \varrho_n} \cW_\varepsilon\bigl(X_{\varrho_n}^{x;u,v}\bigr)
=e^{-\lambda \varrho  _n}{\BY}_{\varrho  _n}^{\e;u,v } ,$ $ \dbP\mbox{-a.s.}
$
On $\{ \varrho >n\}$, we have $ \varrho_n =n$.
Applying  Lemma \ref{LemmaA-1-FBSDE} to \eqref{BY-0-e}, we know
 $|{\BY}_{n}^{\e;u,v}|\les B_\e+\frac{\e B_f+B_\e+B_1+B_2}{\l\e }$.
Combined with
$|\mathcal W_\varepsilon(x)|\les B_\e $, for all $x\in\dbR^n$, we obtain
$$
 \big|
e^{-\lambda n}\cW_\varepsilon\bigl(X_{ n}^{x;u,v}\bigr)
-e^{-\lambda n}{\BY}_{ n}^{\e;u,v }
\big|
\les C_\varepsilon e^{-\lambda n},\q \dbP\mbox{-a.s.}
$$
where $C_\varepsilon$ is independent of $n,u$ and $v$. Thus
\eqref{terminal-difference-rhon} holds.

 Furthermore,  combined with \eqref{Yn-Y},  for every
 $u\in\mathcal U_{0,\varrho}$ and $v\in\mathcal V_{0, \varrho}$, we get
$$      |{\BY}_{0}^{ n,\e ;u,v}- {\BY}_{0}^{\e;u,v }|\les
C_\varepsilon e^{-\lambda n},\q \dbP\mbox{-a.s.}
 $$
Since this estimate is uniform with respect to   $u$ and $v$, by \eqref{G-G} and the definitions of the essential supremum and infimum, we get
$$\ba{ll}
\ns\ds\Big|
\essinf_{\beta\in\mathcal B_{0,\varrho}}
\esssup_{u\in\mathcal U_{0,\varrho}}
G^{ \varrho_n,\varepsilon,x;u,\beta(u)}_{0,\varrho_n}
\big[e^{-\lambda \varrho _n }\cW_\varepsilon\bigl(X_{\varrho_n}^{x;u,\beta(u)}\bigr)\big]
-
\essinf_{\beta\in\mathcal B_{0,\varrho}}
\esssup_{u\in\mathcal U_{0,  \varrho}}
G^{\varrho,\varepsilon,x;u,\b(u)}_{0,\varrho}
\big[e^{-\lambda \varrho }\cW_\varepsilon\bigl(X_{\varrho}^{x;u,\b(u)}\bigr)\big]
\Big|\\
\ns\ds
\les
C_\varepsilon e^{-\lambda n},\q \dbP\mbox{-a.s.}
\ea$$
Letting $n\to\infty$ in \eqref{W-e-re-ex}, we finally  obtain \eqref{auxiliary DPP}.
\end{proof}

  Lemma \ref{DPP} establishes the DPP for a fixed intermediate stopping time.
This form, however, is not sufficient for passing back to the mixed
control--stopping game, because there the intermediate stopping time is
selected through a nonanticipative stopping strategy and therefore may depend
on the opponent's control. We thus need an extension of the DPP to
control-dependent intermediate stopping times.

\begin{proposition} \sl
\label{Lemma-DPP-nonanticipative-stopping}
Assume {\bf (C1)}--{\bf(C5)} hold.

\no(i) Let
$\{\boldsymbol{\rho}^{\beta,\theta}\}_{\beta\in\mathcal B_{0,\infty},
\,\theta\in\mathcal S}$ be a family such that, for each fixed
$\beta\in\mathcal B_{0,\infty}$ and $\theta\in\mathcal S$, the mapping
$
\boldsymbol{\rho}^{\beta,\theta}:\mathcal U_{0,\infty}\to\mathcal S
$
is a nonanticipative stopping strategy. Then, for every $\varepsilon>0$ and $x\in\dbR^n$,
$$
\begin{aligned}
  W_\varepsilon(x)
&\ges
\essinf_{\beta\in\mathcal B_{0,\infty}}
\esssup_{\substack{u\in\mathcal U_{0,\infty}\\ \theta\in\mathcal S}}
G_{0,\boldsymbol{\rho}^{\beta,\theta}[u]}^{\boldsymbol{\rho}^{\beta,\theta}[u],\varepsilon,x;u,\beta(u)}
\left[
e^{-\lambda\boldsymbol{\rho}^{\beta,\theta}[u]}
W_\varepsilon\(
X_{\boldsymbol{\rho}^{\beta,\theta}[u]}^{x;u,\beta(u)}
\)
\right] ,
\qquad \mathbb P\hbox{-a.s.}
\end{aligned}
$$
 (ii)  Let
$\{\boldsymbol{\rho}^{\beta }\}_{\beta\in\mathcal B_{0,\infty}}$ be a family such that, for each fixed
$\beta\in\mathcal B_{0,\infty}$, the mapping
$
\boldsymbol{\rho}^{\beta}:\mathcal U_{0,\infty}\to\mathcal S
$
is a nonanticipative stopping strategy.  Then,  for every $\varepsilon>0$ and $x\in\dbR^n$,
$$
W_\varepsilon(x)
\les \essinf_{\b\in\mathcal B_{0,\infty}}
\esssup_{u\in\mathcal U_{0,\infty}}
G^{\boldsymbol{\rho}^{\beta }[u],\varepsilon,x;u,\beta(u)}_{0,\boldsymbol{\rho}^\beta[u]}
\left[
e^{-\lambda\rho^\beta[u]}
W_\varepsilon\left(
X^{x;u,\beta(u)}_{\boldsymbol{\rho}^\beta[u]}
\right)
\right],
\qquad \mathbb P\text{-a.s.},
$$
 where
the discounted terminal terms are understood to be zero on
$\big\{\boldsymbol{\rho}^{\beta,\theta}[u]=\infty\big\}$ and $\big\{\boldsymbol{\rho}^{\beta}[u]=\infty\big\}$.

\end{proposition}

\begin{proof}
(i)
Fix $\beta\in\mathcal B_{0,\infty}$,
$u\in\mathcal U_{0,\infty}$ and $\theta\in\mathcal S$, and put
$
\boldsymbol{\rho}:=\boldsymbol{\rho}^{\beta,\theta}[u].
$
For any control
$\bar u\in\mathcal U_{\boldsymbol{\rho},\infty}$, define the pasted control
$
u\otimes_{\boldsymbol{\rho}} \bar u
:=
u\,\mathbf 1_{\llbracket 0,\boldsymbol{\rho}\rrbracket }
+
\bar u\,\mathbf 1_{\rrbracket \boldsymbol{\rho},\infty\llbracket}.
$
Then $u\otimes_{\boldsymbol{\rho}}\bar u\in\mathcal U_{0,\infty}$.
By the nonanticipativity of $\beta$, we have
$
\beta(u\otimes_{\boldsymbol{\rho}}\bar u)=\beta(u)
$
on $ \llbracket 0,\boldsymbol{\rho}\rrbracket .
$
Moreover, since $\boldsymbol{\rho}^{\beta,\theta}$ is nonanticipative and
$u\otimes_{\boldsymbol{\rho}}\bar u=u$ on
$\llbracket 0,\boldsymbol{\rho}\rrbracket $, we also have
$$
\boldsymbol{\rho}^{\beta,\theta}
[u\otimes_{\boldsymbol{\rho}}\bar u]\wedge\boldsymbol{\rho}
=
\boldsymbol{\rho}^{\beta,\theta}[u]\wedge\boldsymbol{\rho}
=
\boldsymbol{\rho},
\qquad \mathbb P\hbox{-a.s.}
$$
For
$\bar u\in\mathcal U_{\boldsymbol{\rho},\infty}$, set
$
\beta_{\boldsymbol{\rho}}^u(\bar u)
:=
\left.\beta(u\otimes_{\boldsymbol{\rho}}\bar u)\right|_{\llbracket \boldsymbol{\rho},\infty\llbracket } .
$
Then $\beta_{\boldsymbol{\rho}}^u\in \cB_{\boldsymbol{\rho},\infty}$. Indeed, if
$\bar u^1=\bar u^2$ on $\llbracket \boldsymbol{\rho},\sigma\rrbracket $ for some
$\mathbb F$-stopping time $\sigma\ges\boldsymbol{\rho}$, then
$u\otimes_{\boldsymbol{\rho}}\bar u^1
=
u\otimes_{\boldsymbol{\rho}}\bar u^2$
on $\llbracket 0,\sigma\rrbracket $.
The nonanticipativity of $\beta$ gives
$
\beta(u\otimes_{\boldsymbol{\rho}}\bar u^1)
=
\beta(u\otimes_{\boldsymbol{\rho}}\bar u^2)
$
on $\llbracket 0,\sigma\rrbracket ,
$
and hence
$
\beta_{\boldsymbol{\rho}}^u(\bar u^1)
=
\beta_{\boldsymbol{\rho}}^u(\bar u^2)
$
on $ \llbracket \boldsymbol{\rho},\sigma\rrbracket .
$

\no \textbf{Case 1.}  $\boldsymbol{\rho}<\infty$, $\mathbb P$-a.s.
Applying Lemma \ref{Le-W-e-cY-ran} with
$(\rho,\xi)=(\boldsymbol{\rho},X_{\boldsymbol{\rho}}^{x;u,\beta(u)})$,
we get
$$
 W _\varepsilon ( X_{\boldsymbol{\rho}}^{x;u,\beta(u)})
=\cW _\varepsilon (\boldsymbol{\rho},X_{\boldsymbol{\rho}}^{x;u,\beta(u)})
=
\essinf\limits_{\beta\in\mathcal B_{\boldsymbol{\rho},\infty}}
\esssup\limits_{u\in\mathcal U_{\boldsymbol{\rho},\infty}}
J_\varepsilon(\boldsymbol{\rho},X_{\boldsymbol{\rho}}^{x;u,\beta(u)};u,\beta(u)) .
$$
So for every $\delta>0$ we can choose
$\bar u^\delta\in\mathcal U_{\boldsymbol{\rho},\infty}$ such that
$$
J_\varepsilon
\big(\boldsymbol{\rho},
X_{\boldsymbol{\rho}}^{x;u,\beta(u)};
\bar u^\delta,\beta_{\boldsymbol{\rho}}^u(\bar u^\delta)
\big)
\ges\cW _\varepsilon (\boldsymbol{\rho},X_{\boldsymbol{\rho}}^{x;u,\beta(u)})-\delta
=
W_\varepsilon (X_{\boldsymbol{\rho}}^{x;u,\beta(u)})-\delta,
\qquad \mathbb P\hbox{-a.s.}
$$
Set
$
u^\delta:=u\otimes_{\boldsymbol{\rho}}\bar u^\delta .
$
By the flow property of the backward semigroup
and applying the stability result in Lemma \ref{Pro-A.3} to the BSDE
\eqref{penalized BSDE-auxiliary 3}, we have
$$
\begin{aligned}
&J_\varepsilon(0,x;u^\delta,\beta(u^\delta))
 =
G_{0,\boldsymbol{\rho}}^{\boldsymbol{\rho},\varepsilon,x;u,\beta(u)}
\left[
e^{-\lambda\boldsymbol{\rho}}
J_\varepsilon
 \big(\boldsymbol{\rho},
X_{\boldsymbol{\rho}}^{x;u,\beta(u)};
\bar u^\delta,\beta_{\boldsymbol{\rho}}^u(\bar u^\delta)
 \big)
\right] \\
&  \ges
G_{0,\boldsymbol{\rho}}^{\boldsymbol{\rho},\varepsilon,x;u,\beta(u)}
\[
e^{-\lambda\boldsymbol{\rho}}
W_\varepsilon \big(X_{\boldsymbol{\rho}}^{x;u,\beta(u)} \big)
- \delta e^{-\lambda\boldsymbol{\rho}}\]   \ges
G_{0,\boldsymbol{\rho}}^{\boldsymbol{\rho},\varepsilon,x;u,\beta(u)}
\[
e^{-\lambda\boldsymbol{\rho}}
W_\varepsilon \big(X_{\boldsymbol{\rho}}^{x;u,\beta(u)} \big)
\]
-\delta .
\end{aligned}
$$
Therefore, for each fixed
$\beta\in\mathcal B_{0,\infty}$, $u\in\mathcal U_{0,\infty}$ and
$\theta\in\mathcal S$, we have
$$
\esssup_{u'\in\mathcal U_{0,\infty}}
J_\varepsilon(0,x;u',\beta(u'))
\ges
G_{0,\boldsymbol{\rho}}^{\boldsymbol{\rho},\varepsilon,x;u,\beta(u)}
\[
e^{-\lambda\boldsymbol{\rho}}
W_\varepsilon \big(X_{\boldsymbol{\rho}}^{x;u,\beta(u)} \big)
\]
-\delta.
$$

Since $u\in\mathcal U_{0,\infty}$ and $\theta\in\mathcal S$ were arbitrary,
and notice that $\boldsymbol{\rho}=\boldsymbol{\rho}^{\beta,\theta}[u]$,
we obtain
$$
\begin{aligned}
\esssup_{u'\in\mathcal U_{0,\infty}}
J_\varepsilon(0,x;u',\beta(u'))
&\ges
\esssup_{\substack{u\in\mathcal U_{0,\infty}\\ \theta\in\mathcal S}}
G_{0,\boldsymbol{\rho}^{\beta,\theta}[u]}^{\boldsymbol{\rho}^{\beta,\theta}[u],\varepsilon,x;u,\beta(u)}
\left[
e^{-\lambda\boldsymbol{\rho}^{\beta,\theta}[u]}
W_\varepsilon
\(
X_{\boldsymbol{\rho}^{\beta,\theta}[u]}^{x;u,\beta(u)}
\)
\right]
-\delta .
\end{aligned}
$$
Taking the essential infimum over $\beta\in\mathcal B_{0,\infty}$ gives
$$
W_\varepsilon(x)=\cW_\e(0,x)
 \ges
\essinf_{\beta\in\mathcal B_{0,\infty}}
\esssup_{\substack{u\in\mathcal U_{0,\infty}\\ \theta\in\mathcal S}}
G_{0,\boldsymbol{\rho}^{\beta,\theta}[u]}^{\boldsymbol{\rho}^{\beta,\theta}[u],\varepsilon,x;u,\beta(u)}
\left[
e^{-\lambda\boldsymbol{\rho}^{\beta,\theta}[u]}
W_\varepsilon
\(
X_{\boldsymbol{\rho}^{\beta,\theta}[u]}^{x;u,\beta(u)}
\)
\right]
-\delta .
$$
Finally, letting $\delta\downarrow0$ yields the desired inequality in (i).

\ss

 \no \textbf{Case 2.} $\boldsymbol{\rho}$ may take the value
$\infty$. In this case, set
$
\boldsymbol{\rho}_n:=\boldsymbol{\rho}\wedge n,$ $ n\ges1.
$
Then $\boldsymbol{\rho}_n<\infty$, $\mathbb P$-a.s., and the preceding
argument applies to $\boldsymbol{\rho}_n$. Thus we obtain
$$
\esssup_{u'\in\mathcal U_{0,\infty}}
J_\varepsilon(0,x;u',\beta(u'))
\ges
G_{0,\boldsymbol{\rho}_n}^{\boldsymbol{\rho}_n,\varepsilon,x;u,\beta(u)}
\left[
e^{-\lambda\boldsymbol{\rho}_n}
W_\varepsilon
\big(
X_{\boldsymbol{\rho}_n}^{x;u,\beta(u)}
\big)
\right]
-\delta .
$$
Similar to Step 3 in the proof of Proposition \ref{DPP}, letting $n\to\infty$,
we get
$$
\esssup_{u'\in\mathcal U_{0,\infty}}
J_\varepsilon(0,x;u',\beta(u'))
\ges
G_{0,\boldsymbol{\rho}}^{\boldsymbol{\rho},\varepsilon,x;u,\beta(u)}
\[
e^{-\lambda\boldsymbol{\rho}}
W_\varepsilon
\big(
X_{\boldsymbol{\rho}}^{x;u,\beta(u)}
\big)
\]
-\delta .
$$
Repeating the last part of Case 1, namely taking the essential supremum over
$(u,\theta)$, then the essential infimum over $\beta$, and finally letting
$\delta\downarrow0$, we obtain (i).

\ms

 (ii)  Fix $\beta\in\mathcal B_{0,\infty}$ and put
$
\rho^u:=\boldsymbol{\rho}^{\beta}[u],
$ $ u\in\mathcal U_{0,\infty}.
$
We first consider the case where
$
\rho^u<\infty,$ $ \mathbb P\text{-a.s.},
$ for every $ u\in\mathcal U_{0,\infty}.
$

By Lemma~\ref{Le-W-e-cY-ran}, applied with
$
(\rho,\xi)
=
\bigl(\rho^u,X_{\rho^u}^{x;u,\beta(u)}\bigr),
$
and together with the standard countable-partition approximation for random
initial data, the $\delta$-optimal continuation strategies can be selected as
measurable pastings depending only on the random initial pair
 $
\bigl(\rho^u,X_{\rho^u}^{x;u,\beta(u)}\bigr).
$
In particular, these selections can be made consistently with respect to the
past control path. Therefore, for every $\delta>0$, we may choose, in a
nonanticipative way, a family
$
\big\{
\gamma_{\rho^u}^{u,\delta}\in\mathcal B_{\rho^u,\infty}
:  u\in\mathcal U_{0,\infty}
\big\}
$
such that, for every $u\in\mathcal U_{0,\infty}$,
$$
\esssup_{\bar u\in\mathcal U_{\rho^u,\infty}}
J_\varepsilon\bigl(
\rho^u,X_{\rho^u}^{x;u,\beta(u)};
\bar u,\gamma_{\rho^u}^{u,\delta}(\bar u)
\bigr)
\les
W_\varepsilon\bigl(X_{\rho^u}^{x;u,\beta(u)}\bigr)+\delta,
\qquad \mathbb P\hbox{-a.s.}
$$
Now define a strategy $\beta^\delta\in\mathcal B_{0,\infty}$ as follows,
$$
\beta^\delta(u)
:=
\beta(u)\,{\bf 1}_{\llbracket0,\rho^u\rrbracket}
+
\gamma_{\rho^u}^{u,\delta}\bigl(u|_{\rrbracket\rho^u,\infty\llbracket}\bigr)
\,{\bf 1}_{\rrbracket\rho^u,\infty\llbracket}.
$$
The nonanticipativity of $\beta^\delta$ follows from the corresponding
nonanticipativity properties of $\beta$, the stopping rule
$\boldsymbol{\rho}^{\beta}$, and the  strategies
$\gamma_{\rho^u}^{u,\delta}$.


For each $u\in\mathcal U_{0,\infty}$, using the flow property of  the backward semigroup property
and applying  the stability result in Lemma \ref{Pro-A.3} to the BSDE   \eqref{penalized BSDE-auxiliary 3},  we have
$$
\begin{aligned}
J_\varepsilon(0,x;u,\beta^\delta(u))
&= G_{0,\rho^u}^{\rho^u,\varepsilon,x;u,\beta(u)}
\[
e^{-\lambda\rho^u}
J_\varepsilon\bigl(
\rho^u,X_{\rho^u}^{x;u,\beta(u)};
u|_{\rrbracket\rho^u,\infty\llbracket},
\gamma_{\rho^u}^{u,\delta}
   (u|_{\rrbracket\rho^u,\infty\llbracket})
\bigr)
\]  \\
&\les G_{0,\rho^u}^{\rho^u,\varepsilon,x;u,\beta(u)} \[
e^{-\lambda\rho^u}
W_\varepsilon\bigl(X_{\rho^u}^{x;u,\beta(u)}\bigr)
\]
+ \delta.
\end{aligned}
$$
 Taking the essential
supremum over $u\in\mathcal U_{0,\infty}$ gives
\bel{(ii)-1}
\esssup_{u\in\mathcal U_{0,\infty}}
J_\varepsilon(0,x;u,\beta^\delta(u))
\les
\esssup_{u\in\mathcal U_{0,\infty}}  G_{0,\boldsymbol{\rho}^{\beta}[u]}^{\boldsymbol{\rho}^{\beta}[u],\varepsilon,x;u,\beta(u)}
\[
e^{-\lambda\boldsymbol{\rho}^{\beta}[u]}
W_\varepsilon\(
X_{\boldsymbol{\rho}^{\beta}[u]}^{x;u,\beta(u)}
\)
\]
+ \delta.
\ee
Since $\beta^\delta\in\mathcal B_{0,\infty}$, by the definition of
$\cW_\varepsilon$,
$$
W_\varepsilon(x) =\cW_\e(0,x)
\les
\esssup_{u\in\mathcal U_{0,\infty}}  G_{0,\boldsymbol{\rho}^{\beta}[u]}^{\boldsymbol{\rho}^{\beta}[u],\varepsilon,x;u,\beta(u)}
\[e^{-\lambda\boldsymbol{\rho}^{\beta}[u]}
W_\varepsilon\(
X_{\boldsymbol{\rho}^{\beta}[u]}^{x;u,\beta(u)}
\)
\]
+ \delta.
$$
By the arbitrariness of  $\beta\in\mathcal B_{0,\infty}$, we further obtain
$$
W_\varepsilon(x)
\les
\essinf_{\beta\in\mathcal B_{0,\infty}}
\esssup_{u\in\mathcal U_{0,\infty}}   G_{0,\boldsymbol{\rho}^{\beta}[u]}^{\boldsymbol{\rho}^{\beta}[u],\varepsilon,x;u,\beta(u)}
\[
e^{-\lambda\boldsymbol{\rho}^{\beta}[u]}
W_\varepsilon\(
X_{\boldsymbol{\rho}^{\beta}[u]}^{x;u,\beta(u)}
\)
\]
+ \delta.
$$
Letting $\delta\downarrow0$ gives   the desired inequality in (ii).

\ss

If $\rho^u$ may take the value $\infty$, we set
$
\rho_n^u:=\rho^u\wedge n,\quad n\ge 1.
$
Then $\rho_n^u<\infty$, $\mathbb P$-a.s., and
$u\mapsto\rho_n^u$ remains a nonanticipative stopping strategy.
Applying the preceding argument to $\rho_n^u$ and letting $n\to\infty$,
the same stability estimate of the backward semigroup as in Case 2 of the
proof of (i) yields \eqref{(ii)-1} for $\rho^u$.
Using $\beta^\delta\in\mathcal B_{0,\infty}$, taking the essential infimum
over $\beta$, and then letting $\delta\downarrow0$, we obtain (ii).
\end{proof}

\subsection{Identification with the mixed control--stopping value functions}

In this subsection, we identify the viscosity solution $W$ of the HJBI
variational inequality \eqref{VI-}, obtained in Proposition
\ref{convergence-VI-HJB}, with the lower value function of the mixed-type zero-sum
SDG defined  in \eqref{LU-VF}. To this
end, we first introduce the backward semigroup associated with the original
BSDE \eqref{BSDE}, which provides a convenient tool for the proof.

\bde\label{Def-semi-2}\sl
Given any $x\in\mathbb R^n$, $\dbF$-stopping times $\rho$ and $\varsigma$
with $0\les \rho\les \varsigma$, $\mathbb P$-a.s., and
$\eta\in L^2_{\mathcal F_\varsigma}(\Omega;\mathbb R)$, for
$u\in\mathcal U_{0,\varsigma}$ and $v\in\mathcal V_{0,\varsigma}$, we define
$$
\sG_{\rho,\varsigma}^{x;u,v}\big[e^{-\lambda\varsigma}\eta\big]
:=
e^{-\lambda \rho}\widetilde Y_\rho^{x;u,v},
$$
where $(\widetilde Y^{x;u,v},\widetilde Z^{x;u,v})$ is the solution of the
following BSDE stopped at $\varsigma$:
\begin{equation}\label{y-convergence-v--eq2}\left\{\2n
\begin{aligned}
&\widetilde Y_{s\wedge\varsigma}^{x;u,v}
=
\widetilde Y_{T\wedge\varsigma}^{x;u,v}
+
\int_{s\wedge\varsigma}^{T\wedge\varsigma}
\big(
f\big(
X_r^{x;u,v}, \l\widetilde Y_r^{x;u,v},
\widetilde Z_r^{x;u,v},
u_r,v_r
\big)
-\lambda \widetilde Y_r^{x;u,v}
\big)\md r  \\
&\qq\qq
-
\int_{s\wedge\varsigma}^{T\wedge\varsigma}
\widetilde Z_r^{x;u,v}\,\md B_r,\qq \mbox{ for all }0\les s\les T<\infty, \\
&\widetilde Y_\varsigma^{x;u,v}=\eta
\quad\text{on } \{\varsigma<\infty\}.
\end{aligned}
\right.
\end{equation}
\ede

\begin{proof}[\bf{Proof of Theorem \ref{original problem value function existence}}]
We prove only the lower value case, since the upper value case can be handled similarly.

\noindent\emph{Step 1.} We first prove that $W^-(x)\le W(x)$ for all
$x\in\mathbb R^n$. Fix $x\in\mathbb R^n$, $m\ge1$, $R>|x|$, and
$\beta\in\mathcal B_{0,\infty}$. Define the mappings
$\boldsymbol{\tau}_m^\beta:\mathcal U_{0,\infty}\to\mathcal S$ and
$\boldsymbol{\tau}_R^\beta:\mathcal U_{0,\infty}\to\mathcal S$ by
$$\ba{ll}
\ns\ds
\boldsymbol{\tau}_m^\beta[u]
:=
\inf\Big\{
t\ges0:
W\big(X_t^{x;u,\beta(u)}\big)+\frac1m
\ges
\psi_1\big(X_t^{x;u,\beta(u)}\big)
\Big\},\\
\ns\ds \boldsymbol{\tau}_R^\beta[u]
:=
\inf\left\{
t\ges0:
\big|X_t^{x;u,\beta(u)}\big|\ges R
\right\},
\ea
$$with the convention $\inf\emptyset=\infty$.
Then,
for each $\theta\in\mathcal S$, define
$ \ds
\boldsymbol{\gamma}_m^{\beta,\theta}[u]
:=
\theta\wedge\boldsymbol{\tau}_m^\beta[u],
$ and $\ds
\boldsymbol{\gamma}_{m,R}^{\beta,\theta}[u]
:=
\theta\wedge\boldsymbol{\tau}_m^\beta[u]\wedge
\boldsymbol{\tau}_R^\beta[u].
$
Notice that  $\ds
\Big\{y\in\mathbb R^n: W(y)+\frac1m\ges \psi_1(y)\Big\}
$  is a  closed set.
Since $W$ and $\psi_1$ are continuous and the state process
$X^{x;u,\beta(u)}$ has continuous paths, the hitting time
$\boldsymbol{\tau}_m^\beta[u]\in\cS$ for every fixed
$u\in\mathcal U_{0,\infty}$.    Similarly,
$\boldsymbol{\tau}_R^\beta[u]\in\mathcal S$, and therefore
$\boldsymbol{\gamma}_m^{\beta,\theta}[u],
\boldsymbol{\gamma}_{m,R}^{\beta,\theta}[u]\in\mathcal S$.

Moreover,
$\boldsymbol{\tau}_m^\beta\in\mathfrak T_{0,\infty}$. Indeed, let
$\varrho\in\mathcal S$ and let $u^1,u^2\in\mathcal U_{0,\infty}$ be such that
$u^1=u^2$ on $\llbracket 0,\varrho\rrbracket $. By the nonanticipativity of $\beta$,
$
\beta(u^1)=\beta(u^2)
$ on $ \llbracket 0,\varrho\rrbracket .
$
By the pathwise uniqueness of the controlled SDE,
$$
X^{x;u^1,\beta(u^1)}
=
X^{x;u^2,\beta(u^2)}
\quad\hbox{on } \llbracket 0,\varrho\rrbracket ,
\qquad \mathbb P\hbox{-a.s.}
$$
Consequently, the hitting times determined by these two trajectories coincide
up to $\varrho$, namely
$
\boldsymbol{\tau}_m^\beta[u^1]\wedge\varrho
=
\boldsymbol{\tau}_m^\beta[u^2]\wedge\varrho,
$ $ \mathbb P\hbox{-a.s.}
$
Thus $\boldsymbol{\tau}_m^\beta\in\mathfrak T_{0,\infty}$.
By the same argument,
$\boldsymbol{\tau}_R^\beta\in\mathfrak T_{0,\infty}$. Since $\theta$ is fixed
and does not depend on $u$, taking the minimum with $\theta$ preserves the
nonanticipativity condition. Hence, for every $\theta\in\mathcal S$ and
$R>0$,
$
\boldsymbol{\gamma}_m^{\beta,\theta},
\boldsymbol{\gamma}_{m,R}^{\beta,\theta}
\in\mathfrak T_{0,\infty}.
$

Applying Proposition \ref{Lemma-DPP-nonanticipative-stopping} (i) to the family $\{\boldsymbol{\rho}^{\beta,\theta}\}_{\beta\in\mathcal B_{0,\infty},
\,\theta\in\mathcal S}$ with $ \boldsymbol{\rho}^{\beta,\theta}:=\boldsymbol{\gamma}_{m,R}^{\beta,\theta} $,
we obtain, for every $\varepsilon>0$,
$$
\begin{aligned}
W_\varepsilon(x)
&\ges
\essinf_{\beta\in\mathcal B_{0,\infty}}
\esssup_{\substack{u\in\mathcal U_{0,\infty}\\ \theta\in\mathcal S}}
G_{0,\boldsymbol{\gamma}_{m,R}^{\beta,\theta}[u]}
^{\boldsymbol{\gamma}_{m,R}^{\beta,\theta}[u],\varepsilon,x;u,\beta(u)}
\[
e^{-\lambda\boldsymbol{\gamma}_{m,R}^{\beta,\theta}[u]}
W_\varepsilon
\(
X_{\boldsymbol{\gamma}_{m,R}^{\beta,\theta}[u]}^{x;u,\beta(u)}
\)
\]   =
\essinf_{\beta\in\mathcal B_{0,\infty}}
\esssup_{\substack{u\in\mathcal U_{0,\infty}\\ \theta\in\mathcal S}}
\wt Y _{0}^{ 1, \e;u,\b(u) },
\end{aligned}
$$
where  $\wt Y ^{ 1, \e;u,v }$ satisfies  the following BSDE
%
\begin{equation*}\label{auxiliary penalized backwards semigroups-BSDE}\left\{\2n
\begin{aligned}
&\widetilde Y_{s\wedge\boldsymbol{\gamma}_{m,R}^{\beta,\theta}[u]}^{1,\varepsilon;u,\b(u)}
 =  \widetilde Y_{T\wedge\boldsymbol{\gamma}_{m,R}^{\beta,\theta}[u]}^{1,\varepsilon;u,\b(u)}
+
\int_{s\wedge\boldsymbol{\gamma}_{m,R}^{\beta,\theta}[u]}
^{T\wedge\boldsymbol{\gamma}_{m,R}^{\beta,\theta}[u]}
\Big(
f\big(
X_r^{x;u,\b(u)}, \l\widetilde Y_r^{1,\varepsilon;u,\b(u)},
\widetilde Z_r^{1,\varepsilon;u,\b(u)},
u_r,\b(u)_r
\big)
\\
&\qq\qq\q
+
Q_{\varepsilon}^{W_\varepsilon}
\big(
X_r^{x;u,\b(u)}
\big)
-
\lambda
\widetilde Y_r^{1,\varepsilon;u,\b(u)}
\Big) \md r
-
\int_{s\wedge\boldsymbol{\gamma}_{m,R}^{\beta,\theta}[u]}
^{T\wedge\boldsymbol{\gamma}_{m,R}^{\beta,\theta}[u]}
\widetilde Z_r^{1,\varepsilon;u,\b(u)}\,\md B_r,
\quad
0\les s\les T<\i, \\
&\widetilde Y_{\boldsymbol{\gamma}_{m,R}^{\beta,\theta}[u]}^{1,\varepsilon;u,\b(u)}
 = W_{\e} \(
X_{\boldsymbol{\gamma}_{m,R}^{\beta,\theta}[u]}^{x;u,\beta(u)}
\)
\quad\text{on } \big\{\boldsymbol{\gamma}_{m,R}^{\beta,\theta}[u]<\infty\big\}.
\end{aligned}\right.
\end{equation*}


Since $W_\varepsilon\to W$ locally uniformly on $\mathbb R^n$, for each fixed
$m\ges1$ and $R>0$, there exists $\varepsilon_{m,R}>0$ such that, for all
$0<\varepsilon<\varepsilon_{m,R}$,
$
\sup\limits_{|y|\les R}|W_\varepsilon(y)-W(y)|\les\frac1m .
$
For $0\les t< \boldsymbol{\gamma}_{m,R}^{\beta,\theta}[u] $, we have
$t< \boldsymbol{\tau}_m^\beta[u] $ and $t< \boldsymbol{\tau}_R^\beta[u] $.
Therefore, for $0<\varepsilon<\varepsilon_{m,R}$,
$$
 W_\varepsilon(X_t^{x;u,\beta(u)})
-
\psi_1(X_t^{x;u,\beta(u)})  \les
W_\varepsilon(X_t^{x;u,\beta(u)})
-
W(X_t^{x;u,\beta(u)})
-\frac1m
\les0,
$$
for $\md t\otimes \md \mathbb P$-a.e. on
 $ \llbracket 0,\boldsymbol{\gamma}_{m,R}^{\beta,\theta}[u]\rrbracket $.
Consequently,  $(W_\varepsilon(X_t^{x;u,\beta(u)})
-
\psi_1(X_t^{x;u,\beta(u)}))^+=0$, $\md t\otimes \md \mathbb P$-a.e. on
   $ \llbracket 0,\boldsymbol{\gamma}_{m,R}^{\beta,\theta}[u]\rrbracket  $.
Therefore, $Q_{\varepsilon}^{W_\varepsilon}(X_t^{x;u,\beta(u)})\ges 0$, $\md t\otimes \md\mathbb P$ -a.e. on
$\llbracket0,\gamma_{m,R}^{\beta,\theta}[u]\rrbracket$.
Hence, we  apply  the comparison theorem
to get
\begin{equation}\label{Inequ-1}
\ba{ll}
\ns\ds
  W_{\e}(x)
\ges   \essinf_{\beta\in\mathcal B_{0,\infty}}
\esssup_{\substack{u\in\mathcal U_{0,\infty}\\ \theta\in\mathcal S}}
\wt Y _{0}^{ 2, \e;u,\b(u) }
  =\essinf_{\beta\in\mathcal B_{0,\infty}}
\esssup_{\substack{u\in\mathcal U_{0,\infty}\\ \theta\in\mathcal S}} \sG_{0,\boldsymbol{\gamma}_{m,R}^{\beta,\theta}[u]}^{x;u,\b(u)}\[e^{-\l \boldsymbol{\gamma}_{m,R}^{\beta,\theta}[u]}
W_\e\(X_{\boldsymbol{\gamma}_{m,R}^{\beta,\theta}[u]}^{x,u,\b(u)}\)\] , \   \mathbb{P}\mbox{-a.s.},
\ea\end{equation}
where the semigroup is the one in Definition \ref{Def-semi-2} and   $\wt Y  ^{ 2, \e;u,\b(u) }$ satisfies
\begin{equation}\label{auxiliary penalized BSDE-game-3-1}
 \left\{\2n \begin{aligned}
 &\wt Y _{s\wedge\boldsymbol{\gamma}_{m,R}^{\beta,\theta}[u]}^{2,\e;u,\beta(u)}
=\wt Y _{T \wedge\boldsymbol{\gamma}_{m,R}^{\beta,\theta}[u]}^{2,\e;u,\beta(u)}
+
\int_{s\wedge\boldsymbol{\gamma}_{m,R}^{\beta,\theta}[u]} ^{T\wedge\boldsymbol{\gamma}_{m,R}^{\beta,\theta}[u]} \big(f(X_{r}^{x;u,\beta(u)}, \l\wt Y _{r}^{2,\e;u,\beta(u)}, \wt Z _{r}^{2,\e;u,\beta(u)},u_r,\beta(u)_r) \\
&\qq\qq\q  -  \lambda
\wt Y _{r}^{2,\e;u,\beta(u)} \big)\md r  -\int_{s\wedge\boldsymbol{\gamma}_{m,R}^{\beta,\theta}[u]} ^{T\wedge\boldsymbol{\gamma}_{m,R}^{\beta,\theta}[u]}  \wt Z _{r}^{2,\e;u,\beta(u)}\md B_r,\q \forall 0\les s\les T<\i ,\\
&\wt Y _{\boldsymbol{\gamma}_{m,R}^{\beta,\theta}[u]}^{2,\e;u,\beta(u)}= W_\varepsilon
\(
X_{\boldsymbol{\gamma}_{m,R}^{\beta,\theta}[u]}^{x;u,\beta(u)}
\)\quad\text{on } \big\{\boldsymbol{\gamma}_{m,R}^{\beta,\theta}[u]<\infty\big\}.
    \end{aligned}
    \right.
      \end{equation}

Then, consider the following BSDE,
 \begin{equation}\label{auxiliary penalized BSDE-game-3}
  \left\{\2n\begin{aligned}
 &\wt Y _{s\wedge\boldsymbol{\gamma}_{m,R}^{\beta,\theta}[u]}^{3;u,\beta(u)}
=\wt Y _{T\wedge\boldsymbol{\gamma}_{m,R}^{\beta,\theta}[u]}^{3;u,\beta(u)} +  \int_{s\wedge\boldsymbol{\gamma}_{m,R}^{\beta,\theta}[u]}^{T\wedge\boldsymbol{\gamma}_{m,R}^{\beta,\theta}[u]}
\big(f(X_{r}^{x;u,\beta(u)}, \l\wt Y _{r}^{3;u,\beta(u)}, \wt Z _{r}^{3;u,\beta(u)},u_r,\beta(u)_r) -  \lambda
\wt Y _{r}^{3;u,\beta(u)} \big)\md r    \\
&\qq\qq -\int_{s\wedge\boldsymbol{\gamma}_{m,R}^{\beta,\theta}[u]}^{T\wedge\boldsymbol{\gamma}_{m,R}^{\beta,\theta}[u]} \wt Z _{r}^{3;u,\beta(u)}\md B_r,  \q \forall 0\les s\les T<\i ,\\
&\wt Y _{\boldsymbol{\gamma}_{m,R}^{\beta,\theta}[u]}^{3;u,\beta(u)}=
W
\(
X_{\boldsymbol{\gamma}_{m,R}^{\beta,\theta}[u]}^{x;u,\beta(u)}
\)\quad\text{on } \big\{\boldsymbol{\gamma}_{m,R}^{\beta,\theta}[u]<\infty\big\}.
    \end{aligned}\right.
      \end{equation}
 For fixed $m$ and $R$, since the two BSDEs  \eqref{auxiliary penalized BSDE-game-3-1} and \eqref{auxiliary penalized BSDE-game-3}  have the same driver and differ only in the terminal
condition, Lemma \ref{Pro-A.3} yields
$$\ba{ll}
\ns\ds|\wt Y _{0}^{2,\e;u,v}  -\wt Y _{0}^{3;u,v}|= \Big|\sG_{0,\boldsymbol{\gamma}_{m,R}^{\beta,\theta}[u]}^{x;u,v}\[e^{-\l \boldsymbol{\gamma}_{m,R}^{\beta,\theta}[u]}W_\e\(X_{\boldsymbol{\gamma}_{m,R}^{\beta,\theta}[u]}^{x,u,v}\) \]
-
\sG_{0,\boldsymbol{\gamma}_{m,R}^{\beta,\theta}[u]}^{x;u,v}\[e^{-\l \boldsymbol{\gamma}_{m,R}^{\beta,\theta}[u]}W \(X_{\boldsymbol{\gamma}_{m,R}^{\beta,\theta}[u]}^{x,u,v}\) \]  \Big|\\
\ns\ds \les   \dbE^\dbQ\[\Big|
e^{-\lambda\boldsymbol{\gamma}_{m,R}^{\beta,\theta}[u]}W_\varepsilon
\(X_{\boldsymbol{\gamma}_{m,R}^{\beta,\theta}[u]}^{x;u,\beta(u)}\)
-
e^{-\lambda\boldsymbol{\gamma}_{m,R}^{\beta,\theta}[u]}W
\(X_{\boldsymbol{\gamma}_{m,R}^{\beta,\theta}[u]}^{x;u,\beta(u)}\)
\Big|\cd \mathbf 1_{\{\boldsymbol{\gamma}_{m,R}^{\beta,\theta}[u] <\infty\}} \]\\
\ns\ds \les
\sup_{|y|\les R}|W_\varepsilon(y)-W(y)|\to0,\q \mbox{as }\e\to0.

\ea$$
Therefore,   taking  $\e\rightarrow0$ in \eqref{Inequ-1}, we obtain
 \begin{equation}\label{W-geq-G}
W(x)  \ges  \essinf_{\beta\in\mathcal B_{0,\infty}}
\esssup_{\substack{u\in\mathcal U_{0,\infty}\\ \theta\in\mathcal S}} \sG_{0,\boldsymbol{\gamma}_{m,R}^{\beta,\theta}[u]}^{x;u,\b(u)}\[e^{-\l \boldsymbol{\gamma}_{m,R}^{\beta,\theta}[u]}
W\(X_{\boldsymbol{\gamma}_{m,R}^{\beta,\theta}[u]}^{x,u,\b(u)}\)\] , \q \mathbb{P}\mbox{-a.s.}
 \end{equation}

Next, we consider
 \begin{equation}\label{auxiliary penalized BSDE-game-3-2}
 \left\{\2n \begin{aligned}
 &\wt Y _{s\wedge \boldsymbol{\gamma}_{m}^{\beta,\theta}[u]}^{4;u,\beta(u)}
=\wt Y _{T\wedge \boldsymbol{\gamma}_{m}^{\beta,\theta}[u]}^{4;u,\beta(u)} +  \int_{s\wedge \boldsymbol{\gamma}_{m}^{\beta,\theta}[u]}
^{T\wedge\boldsymbol{\gamma}_{m}^{\beta,\theta}[u]}
\big(f(X_{r}^{x;u,\beta(u)}, \l\wt Y _{r}^{4;u,\beta(u)}, \wt Z _{r}^{4;u,\beta(u)},u_r,\beta(u)_r) -  \lambda
\wt Y _{r}^{4;u,\beta(u)} \big)\md r    \\
&\qq\qq -\int_{s\wedge \boldsymbol{\gamma}_{m}^{\beta,\theta}[u]}
^{T\wedge\boldsymbol{\gamma}_{m}^{\beta,\theta}[u]} \wt Z _{r}^{4;u,\beta(u)}\md B_r,\q  \forall 0\les s\les T<\i   ,\\
&\wt Y _{ \boldsymbol{\gamma}_{m}^{\beta,\theta}[u]}^{4;u,\beta(u)}=W \(X_{\boldsymbol{\gamma}_{m}^{\beta,\theta}[u]}^{x,u,\beta(u)}\) \quad\text{on } \big\{\boldsymbol{\gamma}_{m}^{\beta,\theta}[u]<\infty\big\}.
    \end{aligned}\right.
      \end{equation}
For BSDEs \eqref{auxiliary penalized BSDE-game-3} and \eqref{auxiliary penalized BSDE-game-3-2},  applying  Lemma
\ref{Cor-terminal-time-stability} on
$\llbracket 0,\gamma_{m,R}^{\beta,\theta}[u]\rrbracket$, there exists a probability measure
$\mathbb Q^R$, locally equivalent to $\mathbb P$, such that
$$\ba{ll}
\ns\ds
\big|\wt Y _{0}^{3;u,\beta(u)} - \wt Y _{0}^{4;u,\beta(u)}\big|=\Big|\sG_{0,\boldsymbol{\gamma}_{m,R}^{\beta,\theta}[u]}^{x;u,\b(u)}\[e^{-\l \boldsymbol{\gamma}_{m,R}^{\beta,\theta}[u]}
W\Big(X_{\boldsymbol{\gamma}_{m,R}^{\beta,\theta}[u]}^{x,u,\b(u)}\Big)\]
-\sG_{0,\boldsymbol{\gamma}_{m}^{\beta,\theta}[u]}^{x;u,\b(u)}\[e^{-\l \boldsymbol{\gamma}_{m}^{\beta,\theta}[u]}
W\Big(X_{\boldsymbol{\gamma}_{m}^{\beta,\theta}[u]}^{x,u,\b(u)}\Big)\] \Big|\\
\ns\ds \les C \mathbb E^{\mathbb Q^R}
\Big[
e^{-\lambda\boldsymbol{\gamma}_{m,R}^{\beta,\theta}[u]}\cd
\Big|
W\Big(X_{\boldsymbol{\gamma}_{m,R}^{\beta,\theta}[u]}^{x;u,\beta(u)}\Big)
-
\wt Y _{ \boldsymbol{\gamma}_{m,R}^{\beta,\theta}[u]}^{4;u,\beta(u)}
\Big|\cd \mathbf 1_{\{\boldsymbol{\gamma}_{m,R}^{\beta,\theta}[u] <\infty\}}
\Big].
\ea$$
Since
$\boldsymbol{\gamma}_{m,R}^{\beta,\theta}[u]=\boldsymbol{\gamma}_m^{\beta,\theta}[u]\wedge \boldsymbol{\tau}_R^\beta[u] $,
we have
$
W\(X_{\boldsymbol{\gamma}_{m,R}^{\beta,\theta}[u]}^{x;u,\beta(u)}\)
=
\widetilde Y_{\boldsymbol{\gamma}_{m,R}^{\beta,\theta}[u]}^{4;u,\beta(u)}
$
 on  $\{\boldsymbol{\gamma}_m^{\beta,\theta}[u]\les  \boldsymbol{\tau}_R^\beta[u]\}.
$
On the other hand, on $\{\boldsymbol{\gamma}_m^{\beta,\theta}[u]> \boldsymbol{\tau}_R^\beta[u]\}$, the boundedness of
$W$ and $\widetilde Y^{4;u,\beta(u)}$ yields
$$
 \mathbb E^{\mathbb Q^R}
\[e^{- \lambda\boldsymbol{\gamma}_{m,R}^{\beta,\theta}[u]}
\Big|
W\Big(X_{\boldsymbol{\gamma}_{m,R}^{\beta,\theta}[u]}^{x;u,\beta(u)}\Big)
-
\widetilde Y_{\boldsymbol{\gamma}_{m,R}^{\beta,\theta}[u]}^{4;u,\beta(u)}
\Big| \]
\les C\mathbb E^{\mathbb Q^R}  \[
e^{-\lambda \boldsymbol{\tau}_R^\beta[u]}
{\bf 1}_{\{\boldsymbol{\tau} _R^\beta[u]<\boldsymbol{\gamma}_{m}^{\beta,\theta}[u]\}}
\]   .
$$
Moreover, for any $T>0$,
\begin{equation}\label{ineq-1}
\begin{aligned}
&\mathbb E^{\mathbb Q^R}  \[
e^{-\lambda \boldsymbol{\tau} _R^\beta[u]}
{\bf 1}_{\{\boldsymbol{\tau} _R^\beta[u]<\boldsymbol{\gamma}_{m}^{\beta,\theta}[u]\}}
\]=\mathbb E^{\mathbb Q^R} \[
e^{-\lambda \boldsymbol{\tau}_R^\beta[u]}
{\bf 1}_{\{A\cap\{\boldsymbol{\tau}_R^\beta[u]>T\}\}}
+
e^{-\lambda \boldsymbol{\tau}_R^\beta[u]}
{\bf 1}_{\{A\cap\{\boldsymbol{\tau}_R^\beta[u]\les T\}\}}
\]\\
&\les
e^{-\lambda T}
+
\mathbb Q^R(\boldsymbol{\tau} _R^\beta[u]\les T) \les
e^{-\lambda T}
+
\mathbb Q^R\(
\sup_{0\les s\les T}
\big|X_s^{x;u,\beta(u)}\big|
\ges R
\)   \\
& \les
e^{-\lambda T}
+
\frac{1}{R^2}
\mathbb E^{\mathbb Q^R} \[
\sup_{0\les s\les T}
\big|X_s^{x;u,\beta(u)}\big|^2
\]\les
e^{-\lambda T}
+
\frac{C_T(1+|x|^2)}{R^2}.
\end{aligned}
\end{equation}
The last inequality follows from the uniform estimate
$$
\sup_{u,\beta,R}
\mathbb E^{\mathbb Q^R}
\[
\sup_{0\les s\les T}
|X_s^{x;u,\beta(u)}|^2
\]
\les C_T(1+|x|^2).
$$
Indeed, by Remark \ref{rem:measure-representation}, although $\mathbb Q^R$ may depend on
$R$ and on $(u,\beta,\theta)$, its Girsanov kernel $\zeta^R$ satisfies
$
|\zeta_r^R|\les L_z, $ $ r\ges0.
$
Hence, under $\mathbb Q^R$,
$\ds
B_r^R:=B_r-\int_0^r\zeta_\ell^R\,\md\ell
$  is a Brownian motion on each finite
interval, and the state equation becomes
$$
\md X_s^{x;u,\beta(u)}
=
\[
b\big(X_s^{x;u,\beta(u)},u_s,\beta(u)_s\big)
+
\sigma\big(X_s^{x;u,\beta(u)},u_s,\beta(u)_s\big)\zeta_s^R
\]\md s
+
\sigma\big(X_s^{x;u,\beta(u)},u_s,\beta(u)_s\big)\md B_s^R.
$$
The boundedness of $\zeta^R$ preserves the linear growth condition uniformly in
$u,\beta$ and $R$. The standard finite horizon SDE estimate then gives the
above estimate, with $C_T$ independent of $u,\beta$ and $R$.

Letting first $R\to\infty$ and then $T\to\infty$ in \eqref{ineq-1}, we obtain
$$
\sup_{u,\beta,\theta}
\mathbb E^{\dbQ^R}\left[
e^{-\lambda \tau_R^\beta[u]}
{\bf 1}_{\{\tau_R^\beta[u]<\boldsymbol{\gamma}_{m}^{\beta,\theta}[u]\}}
\right]
\longrightarrow 0.
$$
Therefore, as $R\to\infty$,
$$
\sup_{u,\beta,\theta}
\Big|\sG_{0,\boldsymbol{\gamma}_{m,R}^{\beta,\theta}[u]}^{x;u,\beta(u)}
\Big[
e^{-\lambda\boldsymbol{\gamma}_{m,R}^{\beta,\theta}[u]}
W\Big(X_{\boldsymbol{\gamma}_{m,R}^{\beta,\theta}[u]}^{x;u,\beta(u)}\Big)
\Big]
-
\sG_{0,\boldsymbol{\gamma} _m^{\beta,\theta}[u]}^{x;u,\beta(u)}
\Big[
e^{-\lambda\boldsymbol{\gamma}_m^{\beta,\theta}[u]}
W\Big(X_{\boldsymbol{\gamma}_m^{\beta,\theta}[u]}^{x;u,\beta(u)}\Big)
\Big]\Big|\to0.
$$

 Passing to the limit in \eqref{W-geq-G}, we get
 \begin{equation}\label{5.45}
W(x)  \ges  \essinf_{\beta\in\mathcal B_{0,\infty}}
\esssup_{\substack{u\in\mathcal U_{0,\infty}\\ \theta\in\mathcal S}} \sG_{0,\boldsymbol{\gamma}_{m}^{\beta,\theta}[u]}^{x;u,\b(u)}\[e^{-\l \boldsymbol{\gamma}_{m}^{\beta,\theta}[u]}
W\(X_{\boldsymbol{\gamma}_{m}^{\beta,\theta}[u]}^{x,u,\b(u)}\)\] , \q \mathbb{P}\mbox{-a.s.}
 \end{equation}
Since $\psi_2\les W\les \psi_1$, by the definition
$\boldsymbol{\gamma}_m^{\beta,\theta}[u]
=
\theta\wedge\boldsymbol{\tau}_m^\beta[u]$, we have
$$
\begin{aligned}
&e^{-\lambda\boldsymbol{\gamma}_m^{\beta,\theta}[u]}
W\(
X_{\boldsymbol{\gamma}_m^{\beta,\theta}[u]}^{x;u,\beta(u)}
\)
  =
e^{-\lambda\theta}
W\(
X_\theta^{x;u,\beta(u)}
\)
\mathbf 1_{\{\theta<\boldsymbol{\tau}_m^\beta[u]\}}
+
e^{-\lambda\boldsymbol{\tau}_m^\beta[u]}
W\(
X_{\boldsymbol{\tau}_m^\beta[u]}^{x;u,\beta(u)}
\)
\mathbf 1_{\{\theta\ges\boldsymbol{\tau}_m^\beta[u]\}}       \\
&\ges
e^{-\lambda\theta}
\psi_2\big(
X_\theta^{x;u,\beta(u)}
\big)
\mathbf 1_{\{\theta<\boldsymbol{\tau}_m^\beta[u]\}}  +
e^{-\lambda\boldsymbol{\tau}_m^\beta[u]}
\(
\psi_1\big(
X_{\boldsymbol{\tau}_m^\beta[u]}^{x;u,\beta(u)}
\big)
-\frac1m
\)
\mathbf 1_{\{\theta\ges\boldsymbol{\tau}_m^\beta[u]\}} \\
& =   e^{-\l\boldsymbol{\gamma}_m^{\beta,\theta}[u] }  Y_{\boldsymbol{\gamma}_m^{\beta,\theta}[u] }^{x;u , \theta,\b(u), \boldsymbol{\tau}_m^\beta[u]}
-e^{-\l \boldsymbol{\tau}_m^\beta[u]  }   \frac 1m  {\bf 1}_{\{ \th\ges \boldsymbol{\tau}_m^\beta[u] \}} , \q \mathbb{P}\mbox{-a.s.},
\end{aligned}
$$
where  $Y ^{x;u , \theta,\b(u), \boldsymbol{\tau}_m^\beta[u]} $ is the first part of  the solution of the original BSDE  \eqref{BSDE}.
Then using   Lemma  \ref{Pro-A.3}, we get
\begin{equation*}\label{ }
  \ba{ll}
  \ns\ds
W(x)  \ges \essinf_{\beta\in\mathcal B_{0,\infty}}
\esssup_{\substack{u\in\mathcal U_{0,\infty}\\ \theta\in\mathcal S}}
 \sG_{0,\boldsymbol{\gamma}_m^{\beta,\theta}[u]}^{x;u,\b(u)}   \[ e^{-\l\boldsymbol{\gamma}_m^{\beta,\theta}[u] }  Y_{\boldsymbol{\gamma}_m^{\beta,\theta}[u] }^{x;u , \theta,\b(u), \boldsymbol{\tau}_m^\beta[u]}
-e^{-\l \boldsymbol{\tau}_m^\beta[u]  }   \frac 1m  {\bf 1}_{\{ \th\ges \boldsymbol{\tau}_m^\beta[u] \}} \] \\
 \ns\ds\ges\essinf_{\beta\in\mathcal B_{0,\infty}}
\esssup_{\substack{u\in\mathcal U_{0,\infty}\\ \theta\in\mathcal S}}
 \sG_{0,\boldsymbol{\gamma}_m^{\beta,\theta}[u]}^{x;u,\b(u)}  \[  e^{-\l\boldsymbol{\gamma}_m^{\beta,\theta}[u] }  Y_{\boldsymbol{\gamma}_m^{\beta,\theta}[u] }^{x;u , \theta, \b(u), \boldsymbol{\tau}_m^\beta[u]} \]
 -\dbE\[\Big|e^{-\l(\boldsymbol{\gamma}_m^{\beta,\theta}[u]+ \boldsymbol{\tau}_m^\beta[u] ) }   \frac 1m  {\bf 1}_{\{ \th\ges \boldsymbol{\tau}_m^\beta[u]\} }\Big|\]\\
 \ns\ds\ges\essinf_{\beta\in\mathcal B_{0,\infty}}
\esssup_{\substack{u\in\mathcal U_{0,\infty}\\ \theta\in\mathcal S}}
  Y_{0}^{x;u , \theta, \b(u), \boldsymbol{\tau}_m^\beta[u]}-\frac Cm \ges  \essinf_{\substack{\beta\in\mathcal B_{0,\infty}\\
\boldsymbol{\tau}\in\mathfrak T_{0,\infty}}}
\esssup_{\substack{u\in\mathcal U_{0,\infty}\\ \theta\in\mathcal S}}
 Y_{0}^{x;u , \theta, \b(u), \boldsymbol{\tau}[u]}  -\frac Cm \\
  \ns\ds =W^-(x)-\frac Cm , \q \mathbb{P}\mbox{-a.s.},
 \ea\end{equation*}
where we note $\boldsymbol{\tau}_m^\beta\in\mathfrak T_{0,\infty}$ for every
$\beta\in\mathcal B_{0,\infty}$ and  the definition of the lower
value function $
W^-(x)
 $ in \eqref{LU-VF}.
Letting   $m\to\infty$, we obtain
$
W^-(x)\les W(x),$ $ x\in\mathbb R^n.
$

 \ms

\no\emph{Step 2.}   $W^-(x)\ges W(x),$ $ x\in\mathbb R^n$.
We only indicate the argument, since it is parallel to the proof of
Step 1, with Proposition   \ref{Lemma-DPP-nonanticipative-stopping} (ii)   in place of Proposition \ref{Lemma-DPP-nonanticipative-stopping} (i).

Fix arbitrarily $\beta\in\mathcal B_{0,\infty}$ and
$\boldsymbol{\tau}\in\mathcal T_{0,\infty}$. For $m\ges1$ and $R>0$, define
$$
\boldsymbol{\vartheta}_m^\beta[u]
:=
\inf\Big\{
t\ges0:
W\bigl(X_t^{x;u,\beta(u)}\bigr)-\frac1m
\les
\psi_2\bigl(X_t^{x;u,\beta(u)}\bigr)
\Big\},
$$
and
$ \ds
\boldsymbol{\rho}_m^{\beta, \boldsymbol{\tau}}[u]
:=
\boldsymbol{\vartheta}_m^\beta[u]\wedge \boldsymbol{\tau}[u],
$ $\ds
\boldsymbol{\rho}_{m,R}^{\beta, \boldsymbol{\tau}}[u]
:=
\boldsymbol{\vartheta}_m^\beta[u]\wedge \boldsymbol{\tau}[u]\wedge \boldsymbol{\tau}_R^\beta[u].
$
As before, we have $\boldsymbol{\vartheta}_m^\beta ,
\boldsymbol{\rho}_m^{\beta,\boldsymbol{\tau}},
\boldsymbol{\rho}_{m,R}^{\beta,\boldsymbol{\tau}}
\in\mathfrak T_{0,\infty}.
$

Applying Proposition  \ref{Lemma-DPP-nonanticipative-stopping} (ii) to the family
$
\boldsymbol{\rho}^{\beta}[u]
:=
\boldsymbol{\rho}_{m,R}^{\beta,\boldsymbol{\tau}}[u]
$
and then arguing exactly as in the proof of \eqref{Inequ-1}--\eqref{5.45},
we obtain, after letting $\varepsilon\downarrow0$ and then
$R\to\infty$,
$$
W(x)
\les
\essinf_{\beta\in\mathcal B_{0,\infty}}
\esssup_{u\in\mathcal U_{0,\infty}}
G_{0,\boldsymbol{\rho}_m^{\beta,\tau}[u]}^{x;u,\beta(u)}
\[
e^{-\lambda\boldsymbol{\rho}_m^{\beta,\tau}[u]}
W\left(
X_{\boldsymbol{\rho}_m^{\beta,\tau}[u]}^{x;u,\beta(u)}
\right)
\],
\qquad \mathbb P\hbox{-a.s.}
$$
Since
$
\boldsymbol{\rho}_m^{\beta,\tau}[u]
=
\boldsymbol{\vartheta}_m^\beta[u]\wedge \boldsymbol{\tau}[u],
$
and since $\psi_2\les W\les \psi_1$, we have
$$
\begin{aligned}
&e^{-\lambda\boldsymbol{\rho}_m^{\beta,\boldsymbol{\tau}[u]}}
W\(
X_{\boldsymbol{\rho}_m^{\beta,\boldsymbol{\tau}[u]}}^{x;u,\beta(u)}
\)  \les
e^{-\lambda\boldsymbol{\rho}_m^{\beta,\boldsymbol{\tau}[u]}}
Y_{\boldsymbol{\rho}_m^{\beta,\boldsymbol{\tau}[u]}}^{
x;u,\boldsymbol{\vartheta}_m^\beta[u],\beta(u),\boldsymbol{\tau}[u]}
+
e^{-\lambda\boldsymbol{\vartheta}_m^\beta[u]}\frac1m
{\bf 1}_{\{\boldsymbol{\vartheta}_m^\beta[u]<\boldsymbol{\tau}[u]\}},
\qquad \mathbb P\hbox{-a.s.}
\end{aligned}
$$
By the monotonicity and the Lipschitz property of the backward semigroup,
it follows that
$$
W(x)
\les
\essinf_{\beta\in\mathcal B_{0,\infty}}
\esssup_{u\in\mathcal U_{0,\infty}}
Y_0^{x;u,\boldsymbol{\vartheta}_m^\beta[u],\beta(u),\boldsymbol{\tau}[u]}
+\frac{C}{m},
\qquad \mathbb P\hbox{-a.s.}
$$
Since $\boldsymbol{\vartheta}_m^\beta[u]\in\mathcal S$ for every
$u\in\mathcal U_{0,\infty}$, we further get
$$
W(x)
 \les
\essinf_{\beta\in\mathcal B_{0,\infty}}
\esssup_{\substack{u\in\mathcal U_{0,\infty}\\ \theta\in\mathcal S}}
J\bigl(x;u,\theta;\beta(u),\boldsymbol{\tau}[u]\bigr)
+\frac{C}{m},
\qquad \mathbb P\hbox{-a.s.}
$$
Finally, since $\boldsymbol{\tau}\in\mathcal T_{0,\infty}$ was arbitrary, taking the
essential infimum over $\boldsymbol{\tau}\in\mathcal T_{0,\infty}$ gives
$ \ds
W(x)
\les
W^-(x)+\frac{C}{m}.
$
Letting $m\to\infty$, we obtain
$
W(x)\les W^-(x),$ $ x\in\mathbb R^n.
$
Together with the already proved inequality $W^- \les W $, this proves
$W^-(x)=W(x)$ for all $x\in\mathbb R^n$.
\end{proof}

\begin{remark}\label{Remark-nonanticipative-stopping-strategy}\sl
The mixed control--stopping SDG considered in this paper differs from the
mixed control--stopping problems and games studied in
\cite{Kamizono-Morimoto-2002,Morimoto-2003,Ghosh-Rao-2003,
Ghosh-Rao-Sheetal-2009,Ghosh-Rao-2012,Akdim-Ouknine-Turpin-2006,Hu-2020}.
In our formulation, nonanticipative stopping strategies are introduced; see
\eqref{LU-VF}. We emphasize that this is not merely a notational convenience.

Indeed, as shown in the proof of Theorem
\ref{original problem value function existence}, the intermediate stopping
times are not fixed stopping times chosen independently of the controls. For
example, the hitting times
$
\boldsymbol{\tau}_m^\beta[u]
$
and
$
\boldsymbol{\tau}_R^\beta[u],
$
and consequently
$
\boldsymbol{\gamma}_m^{\beta,\theta}[u]
$
and
$
\boldsymbol{\gamma}_{m,R}^{\beta,\theta}[u],
$
are defined through the controlled state process
$X^{x;u,\beta(u)}$. Hence they depend not only on the control $u$, but also on
the opponent's response $\beta(u)$.

This control dependence is essential in the dynamic programming argument. When
two controls coincide up to a stopping time, the nonanticipativity of the
response strategy ensures that the corresponding opponent controls also
coincide up to that time. Consequently, the controlled trajectories and the
hitting times generated by them are compatible under restriction and pasting.
This property is needed to make the stopping-time construction consistent with
the order of optimization in the lower value function.

In the above-mentioned works, stopping times are used in combined
control--stopping or mixed control--stopping formulations. However, the specific
issue that an intermediate stopping time may depend on the opponent's control
through the response strategy is not explicitly treated as a structural part of
the dynamic programming argument. The present formulation makes this dependence
explicit through nonanticipative stopping strategies, which is crucial for the
identification of the viscosity solution with the lower value function.
\end{remark}

\begin{corollary}\label{value function existence}\sl
Assume the Isaacs minimax condition holds, namely, $$
H^+(x,y,p,X)=H^-(x,y,p,X),\qquad (x,y,p,X)\in
\mathbb R^n\times\mathbb R\times\mathbb R^n\times\mathbb S^n .
$$
Then the mixed-type stochastic recursive  differential game  has a value, that is, $W^+\equiv W^-$.
 Moreover, the two HJBI variational
inequalities \eqref{VI+} and \eqref{VI-} coincide, and the  common
unique bounded viscosity solution in $C_b(\mathbb R^n)$ is the value function of
the mixed-type SDG.

\end{corollary}

\section{Conclusions}

In this paper, we have studied stationary fully nonlinear HJBI-type variational
inequalities with bilateral constraints associated with BSDE-generated
recursive payoffs, and established two stochastic representations in an
infinite-horizon two-player zero-sum framework. The first is an augmented-game
representation. By adding two stopping symbols to the control spaces, the
obstacle terms are encoded in the augmented driver, and the problem is reduced
to recursive Isaacs equations without explicit obstacle constraints. The second
is a mixed control--stopping representation. In this formulation, the two
obstacles arise directly from the stopping decisions of the two players, and
the payoff is defined through a BSDE with a random terminal time.

The formulation contains several classical models as special cases,
including optimal stopping problems and combined stochastic control--stopping
problems. However, the simultaneous presence of two controls, two stopping
decisions and BSDE-generated recursive payoffs leads to a genuinely
game-theoretic bilateral obstacle problem. This structure requires
control-dependent nonanticipative stopping strategies and a dynamic programming
argument adapted to the interaction between continuous controls and stopping
decisions.

The two representations play complementary roles. The augmented formulation is
useful from an analytic and numerical viewpoint, since it connects the obstacle
problem with recursive Isaacs equations. The mixed control--stopping formulation
provides a direct game-theoretic interpretation of the bilateral constraints
and is more suitable for applications involving stopping payoffs, such as
cancellation, default, liquidation and early exercise.

A natural continuation is to develop a direct representation through
infinite-horizon reflected BSDEs with two barriers, where the two obstacles are
imposed on the backward component itself. Such an approach is closely related
to the penalization method used here, but it would require a separate analysis
of well-posedness, stability and dynamic programming for reflected BSDEs with
game features. We leave this direction for future research.

\end{document}